\documentclass[a4paper]{amsart}
\usepackage{ifpdf}
\usepackage{graphicx}
\usepackage[all]{xy}
\input {diagxy} 

\def\sp #1{{\mathcal {#1}}}
\def\goth #1{{\mathfrak{ #1}}}

\def\lie #1{{\sp L_{\!#1}}}
\def\pd#1#2{\frac{\partial#1}{\partial#2}}
\def\>#1{{\bf #1}}                
\def\dd.#1{\frac{\partial}{\partial#1}}

\def\Eq#1{{\begin{equation} #1 \end{equation}}}

\def\tr.{{\box{\rm.\small{Tr}}}}

\def\rp{{\mathrm{p}}}

\newtheorem{theorem}{Theorem}

\newtheorem{proposition}{Proposition}
\newtheorem{definition}{Definition}

\def\dem{\noindent\textbf{\emph{Proof. }}}\def\findem{$~\hfill\Box$}

\begin{document}
\title{Dirac Structures and Reduction of Optimal Control Problems with symmetries}

\author{A. Ibort, T. R. de la Pe\~na, R. Salmoni}
\address{Depto. de Matem\'aticas, Univ. Carlos III de
Madrid, Avda. de la Universidad 30, 28911 Legan\'es, Madrid, Spain.}
\email{albertoi@math.uc3m.es, trodrigu@math.uc3m.es, rsalmoni@math.uc3m.es}

\date{\today}
\thanks{This work was partially supported by SIMUMAT project and MTM2007-62478 Research project, Ministry of Science, Spain.}

\begin{abstract}
We discuss the use of Dirac structures to obtain a better understanding of the geometry of a class of optimal control problems and their reduction by symmetries. 
In particular we will show how to extend the reduction of  Dirac structures recently proposed by Yoshimura and Marsden \cite{Yo09} to describe the reduction of a class of optimal control problems with a Lie group of symmetry.   We will prove that, as in the case of reduction of implicit Hamiltonian or Lagrangian systems, the reduction of the variational principle and the reduction of the Dirac structure describing the Pontryagin Maximum Principle first order differential conditions coincide.  Moreover they will also reproduce E. Mart\'{\i}nez Lie algebroids reduction approach \cite{Mr04} to optimal control systems with symmetry.  The geodesic subriemannian  problem considered as an optimal control problem on the Heisenberg group, will be discussed as a simple example illustrating these results.
\end{abstract}

\maketitle

\tableofcontents

\newpage

{\markboth{Dirac structures}{Reduction of optimal control problems}

\section{Introduction}

We will discuss in what follows the reduction of a class of optimal control problems with symmetry by  using the geometry of Dirac structures.  The class of optimal control problems that are suitable to be discussed from this perspective are those that have a nice geometrical description, i.e., such that both the state and control spaces are smooth manifolds (that will be assumed throughout all this paper without boundary).   Moreover the controlled dynamics is given by a vector field depending on the control parameters and the objective functional is defined by means of a local density.   There is a nice geometrical setting to describe such systems that allows for the introduction of auxiliary geometrical structures that clarify and help with the search and analysis of their solutions.   This is the case for instance of the presence of symmetries.  

Consider an optimal control problem such that there is a Lie group acting on the system and such that it preserves both the dynamics and the objective functional.   We will say then that the Lie group is a group of symmetries of the optimal control problem and that the optimal control problem has the given Lie group as a group of symmetries.    The presence of a symmetry group is always a important tool to solve the problem.  The standard way to proceed in the search for solutions of Hamiltonian or Lagrangian systems is to use the constants of the motion to simplify the problem and eventually to solve it.   More generally, various reduction techniques permit to convert the problem at hand into a reduced one on a (family of) smaller spaces that sometimes are much simpler to deal with.   From the solutions of the reduced problems a reconstruction procedure allows to obtain the solutions to the original problem.    

The history of reduction techniques for dynamical systems with symmetry is involved enough to be accurately described here.   Let us just mention the recent contributions by Yoshimura and Marsden to the problem of reduction of variational principles of (possibly implicit) Lagrangian and Hamiltonian systems on tangent or cotangent bundles by using Dirac structures and their reductions with respect to the action of a Lie group.   We can follow the development of these ideas starting with the use of Dirac structures on Lagrangian mechanics \cite{Yo06a}, \cite{Yo06b}, the reduction of the canonical Dirac structure on the cotangent bundle of a Lie group \cite{Yo07a}, the first analysis of Dirac structures for implicit Lagrangian and Hamiltonian systems \cite{Yo07b} and finally the Dirac reduction of cotangent bundles \cite{Yo09} that has been the main inspiration of this paper.     

Pontryaguine Maximum Principle (PMP) for regular geometrical optimal control problems (to be precise the differential first order conditions of PMP) admits a presymplectic formulation that has been systematically developed in a number of papers (see for instance the recent contribution by Barbero et al \cite{Ba07}-\cite{Ba08} and references therein), hence PMP can be discussed from the point of view of Dirac structures by using the canonical Dirac structure defined by the presymplectic form of the system.   As it was stated before, the problem we are facing consists in exploiting the symmetries of the problem to simplify it. One way to do that is to ``reduce'' the system getting rid of the ``redundant'' degrees of freedom of the problem. This could be done in various ways. The first and most obvious one will be to reduce the presymplectic system $(M,\Omega,H)$ by (pre)symplectic reduction. However we need a deeper analysis of the problem of reduction. (Pre)symplectic reduction is obtained by selecting sub manifolds on $M$ and quotienting them out with respect to the residual symmetry on them.  In this way we obtain a reduced system with fewer degrees of freedom whose solutions allow us to reconstruct the actual solutions of the original system. However when we pass to the reduced system the variational structure of the problem, i.e. the objective functional in our case, is lost, and we cannot be sure that the solutions we obtain in this way are really the minimum of the original objective functional. For this reason it would be better to have a reduction procedure, that will allow us to recover both a nice geometrical structure on the reduced system, as well as a reduced variational structure that would allow us to test the properties of the extremals we are computing.

The notion of Dirac structures and their reductions is particularly useful in this sense because as it has been mentioned before, it has been shown by Marsden and Yoshimura that they provide the right geometrical structure that allows to recover the reduced variational principle.
Hence describing the geometry of optimal control problems by using the notion of Dirac structure is useful in this sense because this framework allows us to reduce optimal control problems with symmetry preserving the variational structure.
Consequently in this paper we will take the road started in \cite{Yo09} and we will apply the same principles to deal with geometrical optimal control problems with symmetry.    We will show that the geometrical picture of the PMP implies to consider a bundle over the state space $P$ of the system that will be called the Pontryaguine bundle and that plays the same role that the space $TQ\oplus T^*Q$ for an implicit Lagrangian system over the configuration space $Q$.    We will describe the PMP first order differential conditions as the dynamics defined on the Pontryagin bundle by a Hamiltonian and a canonical Dirac structure and we will consider the problem of their reduction with respect to the action of a symmetry group $G$.   We will obtain the equations describing the reduced PMP in two different ways.  On one side we will derive them from the direct reduction of the variational principle determined by the use of a Lagrange multipliers theorem and, on the other side, we will proceed directly by considering the reduction of the Dirac structure determining the PMP on the Pontryagin's bundle.  The two reductions, which are computed independently, will lead to the same equations, that can be properly called the reduced PMP.    The diagramme \ref{red_dirac} below reflects the various ideas involved in this paper. 

\begin{figure}
\begin{center}
\includegraphics[width=15cm, height=23cm]{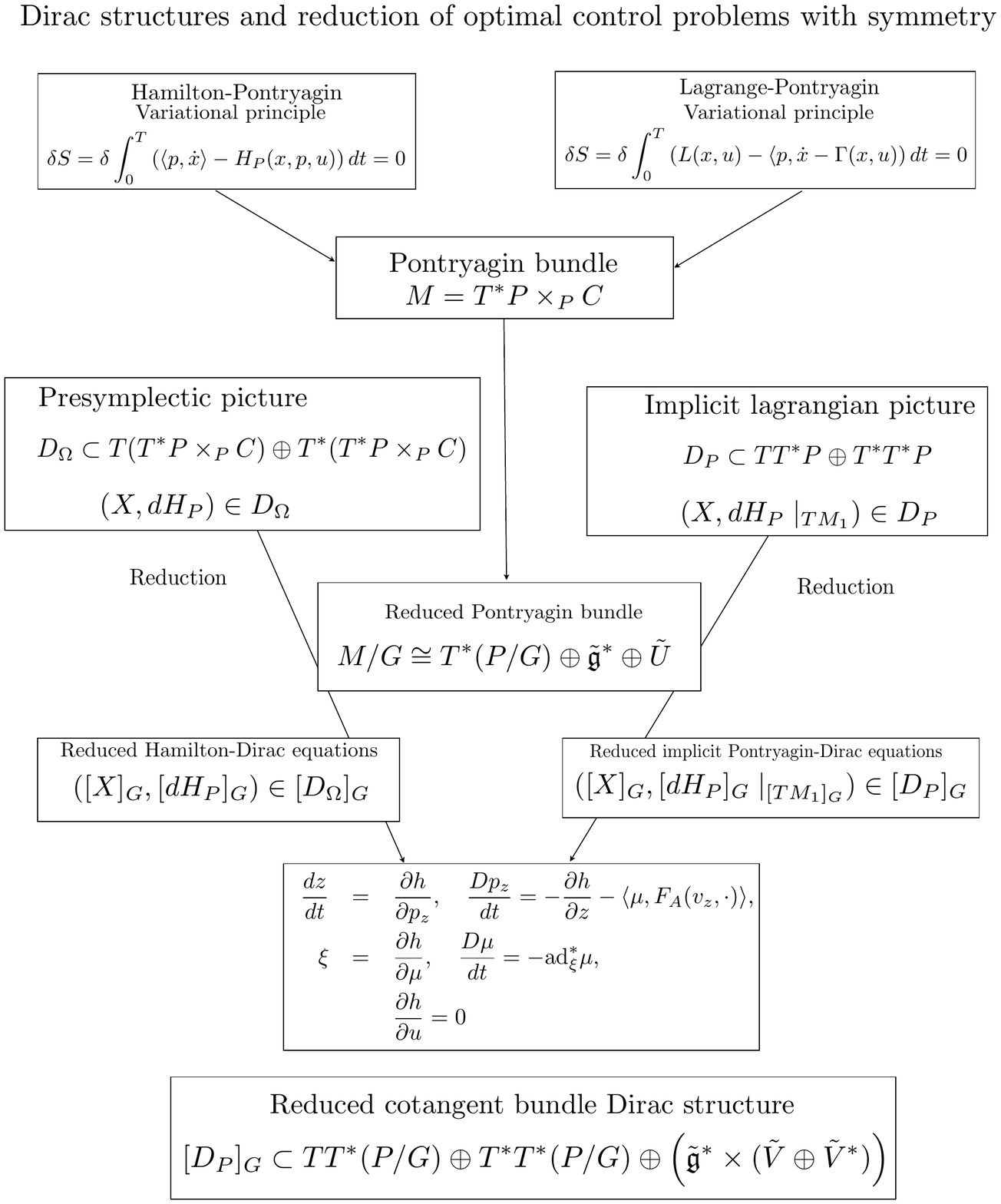}
\caption{Reduction of Dirac structures for optimal control problems} \label{red_dirac}
\end{center}
\end{figure}

Section 2 will be devoted to establish basic notations and terminology as well as the discussion of the geometrical structures involved in the description of the mentioned class of optimal control problems.   The notion of a symmetry group for such systems will be also made precise here and the first discussion on the role of Dirac structures in relation with such systems will be also introduced here.   Section 3 will deal with the computation of the reduced Pontryagin bundle, the reduced variational principle and the corresponding critical points.  The equations describing such critical points will be called the reduced PMP.   Section 4 will address the reduction problem from the Dirac structure perspective, thus the definition of the reduced Dirac structure on the Pontryagin bundle will be discussed and its relation with the cotangent bundle reduction of the canonical Dirac structure on $T^*P$, the space of states and costates of the control problem.   Taking advantage of such relation we will be able to show quite straightforwardly the corresponding reduced equations that will reproduce the reduced PMP above.  Finally, section 4 will illustrate the previous ideas by the explicit computation of the solutions of the geodesic subriemannian problem determined by the Heisenberg group $\Bbb{H}^1$ considered as an optimal control problem on the Heisenberg group with symmetry.

\newpage

\markboth{Dirac structures}{Reduction of optimal control problems}
\section{Geometrical and Variational aspects of a class of optimal control problems with symmetry}

As it was indicated in the introduction we will consider in what follows geometrical optimal control problems.  This implies that we are going to make some assumptions about the structure of the systems in order to make transparent the geometrical ideas involved. 
As it was already indicated all spaces will be assumed to be smooth manifolds without boundary as well as all functions, curves, vector fields and forms will be assumed to be smooth.   In addition we will not address the discussion of abnormal extremals and singular optimal control systems, i.e., we will only consider regular optimal control problems (see below).  Moreover we consider that the class of admissible controls consists of smooth functions.   In doing so we will exhibit clearly the geometrical ideas without paying attention to the analytical difficulties that arise when non-smooth functions are shown.  It must be pointed out that the various constructions can be extended in a more or less natural way to more involved situations where other analytical or geometrical considerations will be needed, for instance in dealing with bounded domains or singular problems.  We will provide now some background to fix the notation.

\subsection{A geometrical picture of optimal control problems}

\begin{definition}[Control System]
A \textsf{control system} is a triple $(P,C,\Gamma)$ where:
\begin{enumerate}
\item[i.] $P$ is a smooth manifold which is the state space of the control system.
Local coordinates in $P$ will be denoted by $x^i$, $i=1....n$ and generic points by $x$.

\item[ii.] $\pi \colon C \to P$ is a locally trivial bundle over $P$ which models the controls, i.e. the fiber $C_x = \pi^{-1}(x) \subset C$ over the state $x\in P$, is the space of controls acting on the system at the state $x$, which is assumed again to be a smooth manifold.  Adapted local coordinates on the control bundle will be denoted by $(x^i,u^a)$, $i=1,\ldots,n$, $a=1,\ldots,r$, and a generic point by $(x,u)$.

\item[iii.] $\Gamma$ is a vector field on $C$ along the projection map $\pi$ which describes the dynamics of the control system,  i.e., $\Gamma \colon C\rightarrow TP$ such that $\tau_P \circ \Gamma = \pi$, where $\tau_P \colon TP \rightarrow P$ denotes the canonical projection map.
$$\begin{array}{cc} { \bfig\xymatrix{&C\ar[dr]^{\pi}\ar[r]^{\Gamma} & TP\ar[d]^{\tau_P}\\   && P&}\efig} \end{array} $$
Natural local coordinates on the tangent bundle $TP$ will be denoted by $(x^i,v^i)$ or $(x^i,\dot{x}^i)$ if there is not risk of confusion.
The local expression of $\Gamma$ will be thus $\Gamma= f^i(x,u) \partial/\partial x^i$ and its integral curves $x^i(t)$ will satisfy the system of differential equations:
\begin{equation} \label{cs}
\dot{x}^i(t)=f^i(x(t),u(t)) , \quad i = 1, \ldots, n ,
\end{equation}
for $u(t)$ belonging to some class of smooth functions  $u: I \subset \mathbb{R}\rightarrow C$.
\end{enumerate}
\end{definition}

\begin{definition}[Optimal Control Problem]   An \textsf{optimal control problem} (OCP for short) is a quadruple $(P,C,\Gamma,L)$ where:
\begin{itemize}
\item[i.]  $(P,C,\Gamma)$ is a control system.
\item[ii.] $L: C \rightarrow \mathbb{R}$ is a smooth function defined on the control bundle, which is the lagrangian density of an objective functional to be minimized along the integral curves of the control system with fixed endpoints $x_0=x(0)$, $x_T=x(T)$.
The objective functional of the problem will have the form:
\begin{equation}
\label{of}
S(\sigma)= \int_0^T L(x(t),u(t))dt ,
\end{equation}
where $\sigma: [0,T] \rightarrow C$ is a smooth curve on $C$ such that $\sigma(t)=(x(t),u(t))$, $x(t)=\pi \circ \sigma (t)$ and $u(t) \in \pi^{-1} (x(t))$, $ \forall t \in [0,T]$.
\end{itemize}
\end{definition}


\subsection{The presymplectic picture of geometric optimal control problems}
\label{PMP}
It is well known that the basic tool for Optimal Control Theory is the Pontryagin Maximum Principle (PMP for short). Here we consider the first order conditions in the PMP and we assume that the minimum of (\ref{of}) exists.
It is natural to give a presymplectic geometrical description of the first order conditions in the PMP.

\begin{definition}[Pontryagin bundle]
We call the \textsf{Pontryagin bundle} over $P$ the manifold $M=T^*P\times_P C$ defined as the bundle over the base space $P$ such that the fiber over $x \in P$ is the cartesian product of the corresponding fibers $T_x^*P$ and $C_x$ over $x \in P$.
\end{definition}

The manifold $M=T^*P \times_P C$  is the fiber cartesian product of the two bundles $T^*P$ and $C$ over $P$ and it has natural local coordinates $(x^i,p_i,u^a)$ induced from local coordinates $(x^i, p_i)$ of the cotangent bundle $T^*P$ and $(x^i, u^a)$ of the control bundle $C$.   We will denote the canonical projection maps from $M$ into the factor $T^*P$ and $C$ by $\rp_P$  and $\rp_C$ respectively, thus:

 $$\rp_P  \colon M \to T^*P, \quad \rp_P (x,p,u)=(x,p)$$ and
 $$\rp_C  \colon M \to C, \quad \rp_C (x,p,u)=(x,u) .$$

$$\bfig\xymatrix{&M= T^*P \times_P C \ar[dl]_{\rp_P}\ar[dr]^{\rp_C}&\\
T^*P\ar[dr]_{\pi_P}&&C\ar[dl]^{\pi}\\ &P&} \efig$$

The total space manifold $M$ carries a natural presymplectic structure $\Omega$ obtained by pulling-back to $M$ the canonical symplectic structure $\omega_P$ in $T^*P$, i.e.,
\begin{equation}\label{presympl}
\Omega=\rp_P ^* \omega_P=dx^i \wedge dp^i .
\end{equation}

Moreover the vector field $\Gamma$ and the lagrangian $L$ combine to define a natural ``energy" function on $M$, namely the Pontryagin Hamiltonian,
\begin{equation}\label{pont_ham}
H_P(x,p,u)=\langle p,\Gamma(x,u) \rangle-L(x,u) = p_if^i(x,u)-L(x,u) .
\end{equation}

The presymplectic system $(M,\Omega,H_P)$ provides a geometrical model for the PMP because the integral curves of $(M,\Omega,H_P)$ are extremals for the optimal control problem defined by $(P,C,\Gamma,L)$.   In fact we can notice that the characteristic distribution $K$ of $\Omega$, $\ker \Omega = K$, is spanned by vertical vector fields with respect to the projection map $\rp_P$, i.e.
 $$K=span\left\{ \dfrac{\partial}{\partial u^a} , a=1,\ldots,r\right\}.$$ 
Thus if we denote by $X_P$ a vector field defined by the presymplectic system above, i.e.
\begin{equation}\label{presym}
i_{X_P} \Omega = dH_P,
\end{equation}
a necessary condition \cite{Go79} for the existence of solutions for the previous implicit differential equation in $M$, i..e. integral curves of a vector field $X_P$ satisfying eq. (\ref{presym}),  is $\langle Z,dH \rangle =0$, $\forall Z \in K$, or in local coordinates:
\begin{equation}
\label{constr}
 \varphi_a = \frac{\partial H_P}{\partial u^a}=0 .
 \end{equation}

The conditions $\varphi_a = 0$, $a=1,\ldots, r$, generically define a sub manifold $M_1$ and if the presymplectic system eq. (\ref{presym}) has solutions they will satisfy the set of differential equations:

\begin{equation}\label{difeq}
\dot{x}^i=\frac{\partial H_P}{\partial p^i} (x,p,u)\qquad
\dot{p}_i= -\frac{\partial H_P}{\partial x^i} (x,p,u)
\end{equation}
on $M_1$. This is equivalent to the first order conditions in the PMP, i.e, the vector fields $X_P$ defined by (\ref{presym}) together with the consistency conditions (\ref{constr}) provide the extremals for the optimal control problem $(P,C,\Gamma,L)$.
In the case that the initial value problem posed by the differential-algebraic equations (\ref{constr})-(\ref{difeq}) at any point $x \in M_1$, has a unique solution we will say that the OCP is regular.   A necessary and sufficient condition for the system $(P,C,\Gamma,L)$ to be regular is that $W_{ab}=\partial^2H_P/\partial u^a\partial u^b$ has maximal rank everywhere on $M_1$.   We just consider regular OCP's in this paper.

On the other hand, an alternative way to think about the dynamics provided by the presymplectic structure $\Omega$ is to consider the complete lifting of the control vector field $\Gamma$ to $M$ in analogy with what happens for a vector field on a manifold.    If $Y$ is a vector field on a manifold $N$, we can consider the complete lifting $Y^c$ of $Y$ to $T^*N$ as the hamiltonian vector field defined by the hamiltonian function $P_Y (x,p) = \langle p, Y(x) \rangle$.   Such vector field $Y^c$ satisfies $i_{Y^c} \omega_Y = dP_Y$ and the flow of $Y^c$ is just the cotangent lifting of the flow of $Y$.   We can add a potential term $L(x)$ to the previous hamiltonian $P_Y$ to obtain a generalized complete lifting $\tilde{Y}$ defined  by $i_{\tilde{Y}}\omega_N = dH$, $H (x,p ) = \langle p, Y(x) \rangle - L(x)$.  Notice that the vector field $\tilde{Y}$ is projectable with respect to the canonical projection map $\pi_N\colon T^*N \to N$ and projects onto $Y$.   This construction can be extended to the case of vector fields along maps.   Thus, if $Y\colon W  \to TN$ is a vector field along the map $\pi \colon W \to N$, we can define its complete lifitng to $T^*N \times_N W = \{ (x,p,w) \in T^*N \times W \mid p \in T_x^*N, \, \pi(w) = x  \}$, as the vector field $Y^c$ on $T^*N\times_N W$ along the map $\rp_N \colon T^*N\times_n W \to T^*N$ given by $\rp_N (x,p,w) = (w,p)$ such that $i_{Y^c} \omega_N = dP_Y$.

 we can consider the vector fields $X$ determined by eq. (\ref{presym}) as a vector field along the projection map $\rp_P\colon M \to T^*P$, i.e., $X \colon M \to T(T^*P)$, such that $\tau_{T^*P}\circ X = \rp_P$.    

\subsection{The Lagrange-Pontryagin principle for optimal control problems}
We provide here a variational picture of an optimal control problem which leads directly to the presymplectic picture above.
Consider the OCP $(P,C,\Gamma,L)$ and denote as before by $M=T^*P \times_P C$ its Pontryagin bundle. We can consider this optimal control problem as a constrained minimization problem for the functional (\ref{of}) with constraints given by the control equations (\ref{cs}), then under appropriate analytical conditions the Lagrange multipliers theorem allow us to characterize the extremals of such problem as the critical points of the functional  $\mathbb{S}: \mathcal{C}_{x_0,x_T}(M)\rightarrow \mathbb{R}$ defined by
\begin{equation}\label{LPvariational}
\mathbb{S}(\gamma)=\int_0 ^T\left[L(x(t),u(t)) + \langle p(t),\dot{x}(t)-\Gamma(x(t),u(t))\rangle \right] \, dt
\end{equation}
where $\mathcal{C}_{x_0,x_T} (M)$ is the space of the smooth curves $\gamma\colon [0,T]\to M$ with fixed endpoints $x(0) = x_0$, $x(T) = x_T$ and $\gamma(t) = (x(t),p(t),u(t))\in \mathcal{C}_{x_0,x_T} (M)$.

Although the solutions for an optimal control problem are the minima of the objective functional, we emphasize that in this geometrical framework we are looking for the extremals of $\mathbb{S}$, assuming that the minimum exists in the class of curves where we are considering.   Then if we denote a tangent vector to the space $\mathcal{C}_{x_0,x_T}(M)$ at the curve $\gamma(t)$ by $\xi \in T_\gamma \mathcal{C}_{x_0,x_T}(M)$, and standard calculation renders (see for instance \cite{De04}):
 $$
 d_\xi \mathbb{S}(\gamma)=\int_0 ^T \langle i_{\dot{\gamma}(t)}\Omega_{\gamma(t)} - d{H_P}(\gamma(t)),\xi(t)\rangle \, dt ,
 $$
for the differential of $\mathbb{S}$ in the direction of $\xi$ at $\gamma$, where  $\Omega$ is the presymplectic structure on $M$, eq. (\ref{presympl}), and $H(x,p,u)$ is the Pontryagine Hamiltonian, eq. (\ref{pont_ham}).

The differential of $\mathbb{S}$ is a local $1$-form on $\mathcal{C}_{x_0,x_T}(M)$ with local density the $1$-form along $\gamma(t)$ given by $i_{\dot{\gamma}}\Omega_{\gamma}-d{H_P}\circ \gamma$, which is precisely the geometrical structure present in the presymplectic picture. Moreover $\gamma$ is an extremal curve for $\mathbb{S}$ iff:

\begin{equation}
\label{excur}
 i_{\dot {\gamma}(t) } \Omega_{\gamma(t)} =dH_P(\gamma(t)),  \qquad t \in [0,T].
\end{equation}
Even if in a mechanical framework a variational principle of this form was called Hamilton-Pontryagine variational principle \cite{Yo09}, in the sequel we will call it \textsf{Lagrange-Pontryagine Principle} for the OCP $(P,C,\Gamma,L)$, because it is obtained from a Lagrangian density and is defined on the Pontryagine bundle.  We will call \textsf{Hamilton-Pontryagine Principle} the corresponding hamiltonian formulation, that is obtained from eq. (\ref{LPvariational}) rearranging the terms in under the integral:
\begin{equation}\label{HPvariational}
\mathbb{S}(\gamma)=\int_0 ^T\left[\langle p(t), \dot{x}(t) \rangle - H_P(x(t),p(t),u(t) \right] \, dt .
\end{equation}
 

\subsection{Dirac structures for optimal control problems}

Once the presymplectic and the variational pictures for optimal control problems have been introduced we proceed further and we will describe their geometry by using Dirac structures.
We recall here the definition of a Dirac structure \cite{CW88}, \cite{Co90}.

\begin{definition}[Dirac structure on a vector bundle $E$]
Let $E\rightarrow N$ be a vector bundle on a smooth manifold $N$, a Dirac structure $D$ on $E$ is a maximally isotropic sub bundle of the bundle $E \oplus E^*$ under the bilinear form $\langle\langle \cdot,\cdot\rangle\rangle$ given by $\langle\langle (u,\alpha),(v,\beta)\rangle\rangle =\langle \beta, u \rangle + \langle \alpha, v\rangle$,  $u,v \in E$ and $\alpha,\beta \in E^*$.
\end{definition}

\begin{definition}[Integrability of a Dirac structure]
If $E=TN\rightarrow N$ is a tangent bundle of a smooth manifold $N$, then we will say that a Dirac structure $D$ on $TN$ (i.e. a maximally isotropic sub bundle $D\subset TN\oplus T^*N$) is integrable if the condition
$$\langle\lie{X_1} \alpha_2, X_3\rangle+
\langle \lie{X_2}\alpha_3, X_1 \rangle+ \langle \lie{X_3}
\alpha_1,X_2 \rangle=0, $$
is satisfied 
for all $(X_1,\alpha_1)$, $(X_2,\alpha_2)$, $(X_3,\alpha_3) \in D$ where 
$\lie{X}$ denotes the Lie derivative along the vector field $X$ on $M$. 
This condition is equivalent to say that $D$ is involutive with respect to the Courant Bracket \cite{Co90} defined on the set of the sections of the bundle $TN \oplus T^*N$:
$$ [(X,\alpha),(Y,\beta)]=([X,Y], \lie{X} \beta - \lie{Y} \alpha + \dfrac{1}{2} d(\beta(X)-\alpha(Y)) ) ,$$
i.e., if it holds
$$
[(X,\alpha),(Y,\beta)] \in D \qquad \forall ((X,\alpha),(Y,\beta))\in D
$$
\end{definition}

\begin{definition}[Dirac structure on a manifold N]
An integrable Dirac structure on $TN$ is called a Dirac structure on the manifold $N$
\end{definition}

A distinguished class of Dirac structures are provided by graphs of both presymplectic or Poisson structures on manifolds.  Thus we define the Dirac structure $D_\Omega$ associated to a given  2-form $\Omega$ on a smooth manifold $N$ as:
$$
D_\Omega =\{ (v,\alpha) \in TN\oplus T^*N \mid \alpha(w)= \Omega(v,w) ~ \forall w\in TN\}
$$
and the closedness of $\Omega$ is equivalent to the integrability of $D_\Omega$.
 
Thus it is natural to consider the Dirac's structure associated with an optimal control problem as the Dirac structure defined by its presymplectic picture.

\begin{definition}[Dirac structure picture of an OCP]  We will call the Dirac structure $D_\Omega \subset TM\oplus T ^*M$ defined by the canonical presymplectic structure $\Omega$ on the Pontryagine bundle $M = T ^*P \times_P C$ the \textsf{Dirac structure of the optimal control problem} $(P,C,\Gamma, L)$.
\end{definition}
 
Locally $D_\Omega$ is given by $D_{\Omega}(x,p,u) = \{(v_x,v_p,v_u;p_x,p_p,p_u)\in
T_{(x,p,u)}M\oplus T_{(x,p,u)}^*M \, | \, v_x=p_p;\, v_p=-p_x;\, p_u=0\}$.

If $D$ is a Dirac structure on a manifold $N$ and $H$ is a smooth function on it, the dynamics defined by $D$ and $H$ are given by the vector fields $X$ on $N$ satisfying: $(X,dH) \in D$.  We can conclude therefore with the proposition-definition for the Dirac picture for the PMP of an optimal control problem.

\begin{proposition}
Let $(P,C,\Gamma,L)$ be an optimal control problem. Consider the canonical Dirac structure on $M=T^*P \times_P C$ given by $D_{\Omega}$ where $\Omega$ is the canonical presymplectic structure on $M$.   If $H_P(x,p,u)=\langle p,\Gamma(x,u)\rangle - L(x,u)$ defines the Pontryagin's Hamiltonian, then the dynamics defined by $D_{\Omega}$ and $dH$ describe the extremal curves for the PMP of the optimal control problem:
$$
(X,dH)\in D_{\Omega} .
$$
\end{proposition}

Because we are interested in regular optimal control problems there is a further refinement of the description of the OCP in terms of Dirac structures that will be very helpful later on.  

$$\begin{array}{cc}
{ \bfig\xymatrix{&M\ar[dr]^{\emph{p}_P}\ar[r]^{X} & T(T^*P)\ar[d]^{\tau_{T^*P}}\\
&&T^*P&}\efig} & {
\bfig\xymatrix{&M\ar[dr]^{\Phi}\ar[r]^{\emph{p}_P} &
T^*P\ar[d]^{\pi_P}\\   && P&}\efig}
\end{array}
$$

\bigskip

Then we can consider the canonical Dirac structure $D_P \subset
T(T^\ast P)\oplus T^\ast (T^\ast P)$ defined by the canonical symplectic structure $\omega_P$ in $T^\ast P$.
We have seen in section (\ref{PMP}) that the equations of the system related to the PMP are described by a vector field $X$ along the canonical map  $\mathsf{p}_P \colon M=T^*P \times_P C \to T^*P$, i.e. $X \colon M \to T(T^\ast P)$ such that $\tau_{T^\ast
P} \circ X= \emph{p}_P$ or locally, if $(x,p,u)\in M$, then
$X(x,p,u) \in T_{x,p}(T^*P)$. The Pontryagine Hamiltonian  $H
\colon M \to \mathbb{R}$ allow us to define $dH \colon M \to
T^*M$ and we can take the restriction of $dH$ to the reduced subspace
$TM_1$, where $i_1 \colon M_1 \to M$ is the sub manifold defined by the consistency conditions of the Gotay-Nester algorithm (\ref{constr}) : $\varphi_a=\pd{H}{u^a}=0$. We stress that $i_{1 \ast}
\colon TM_1 \hookrightarrow TM$ allow us to consider in a natural way the restriction of $dH$ to $TM_1$.
 In this way the implicit Dirac formulation of the Optimal Control Problem $(P,C,\Gamma,L)$ is given by the equations:

$$(X, dH \Big |_{TM_1}) \in D_P ~ ~in ~~ M_1$$

or:
$(x,p,u)\in M$, then $$\left(X(x,p,u), dH(x,p,u) \Big |_{TM_1}
\right)\in D_P(x,p), ~\forall (x,p,u)\in M_1.$$

Notice that this equations restricted to $M_1$ have the same meaning of the hamiltonian equations in the presymplectic picture when we solve the consistency conditions (\ref{constr}).
\subsection{Symmetries and Optimal Control Systems}\label{symmetric_ocp}

The main purpose of this paper is the study of geometrical optimal control systems with symmetries, and in particular, their reduction both from the point of view of the reduction of the variational problem that define them and the reduction of the corresponding PMP.   
For it we will start by specifying the notion of symmetry of an optimal control problem. 

Let $(P,C,\Gamma,L)$ be an optimal control problem. Let $G$ be a Lie group acting on the bundle $C$. We will assume that the action of $G$ on $C$ is fibered with respect to the fiber bundle structure of $C$ over $P$.  We denote the action of $G$ on $C$ by: $$\Psi: G\times C \rightarrow C$$ and the action of $G$ on $P$ by: $$\Phi: G \times P \rightarrow P$$
 Thus we have that $\pi\circ \Psi_g = \Phi_g \circ \pi, ~ \forall g \in
 G$ where $\Phi_g (x)= \Phi(g,x)$
 and $\Psi_g (x,u)=\Psi (g, (x,u)), ~\forall g\in G ,~ (x,u)\in C $
 or, in other words, the following diagram commute:
$$\begin{array}{cc}
{ \bfig\xymatrix{C\ar[r]^{\Psi_g}\ar[d]_{\pi}  & C\ar[d]^{\pi} \\
P\ar[r]_{\Phi_g}&P&&}\efig }
\end{array}
$$
\begin{definition}\label{ocp_symmetry}
The optimal control problem $(P, C, \Gamma, L)$ will be said to be $G$-invariant, or to have $G$ as a group of symmetries, if the Lie group $G$ has a fibered action $(\Phi, \Psi)$ on the bundle $C$ and moreover it satisfies:
\begin{itemize}

\item[i.] $L \circ  \Psi_g=L$, 
\item[ii.] $T\Phi_g \circ \Gamma = \Gamma \circ \Psi_g$,
\end{itemize}
for all $g\in G$.
\end{definition}

First, we must notice that the action $\Phi$ of the group $G$ on $P$ lifts both to $TP$ and $T^*P$ by tangent and cotangent lifting respectively. Thus we shall denote by $g\cdot v =T\Phi_g(v)$ and $g\cdot
p=T\Phi^*_{g^{-1}} (p)$, for all $v\in T_xP$, $p\in T_x^*P, \: g\in G$ the corresponding lifted actions.
Then the group $G$ will also act on the Pontryagine bundle $M= T^*P \times_P C$ as:
 $$ g\cdot (x,p,u)=(gx,gp,gu)=(\Phi_g(x),T\Phi^*_{g^{-1}}(p), \Psi_{g_{x}}(u)) .$$

 Secondly, we observe that the presymplectic picture of the control system is also $G$-invariant and the group $G$ is represented (pre)symplectically on such system.

 \begin{proposition}
 The group $G$ acts (pre)symplectically on $(M,\Omega,H)$.
 Moreover, there is an equivariant momentum map $J_M: M\rightarrow \goth g^*$ such that:
 $$ i_{\xi_M}\Omega=\langle dJ_M,\xi \rangle$$
 where $\xi \in \goth g$ and $\xi_M$ is the vector field on $M$ describing the infinitesimal action of $\xi$.
 Moreover, we have $\rp_P^*J_P = J_M$ , where $J_P\colon T^*P \to \goth g^*$ is the canonical momemtum map corresponding
 to the action of $G$ on $T^*P$.  
 \end{proposition}
 
 \begin{proof}  A simple computation shows that the function $J_M(x,p,u) = \langle p,\xi_M(x,p,u) \rangle$ defined on $M$ satisfies the equation above, hence it is the (presymplectic) momentum map of the action.  Again it is easy to check that $J_M$ is equivariant, i.e.,
 $J_M(gx,gp,gu) = Ad_{g-1}^* J_M(x,p,u)$.
 
 \end{proof}

 In general we will assume for simplicity in what follows that the orbit space of $G$ on $P$, i.e., $P/G$ is a smooth manifold
and that the canonical projection map $\rho \colon P \to P/G$ is a submersion.  Later on in order to work out compact expressions for the expressions describing the reduced optimal control problems, we will even assume that the action of $G$ on $P$ is free and proper, i.e., that $P$ is a principal fiber bundle over $P/G$ with structural group $G$.

 Besides it will be convenient for the stated purposes to assume that the control
 bundle $C$ has a particular form that will make easy to work out the expressions and equations on the reduced spaces.  We will assume thus that $C$ is the pull--back to $P$ of an associated space to $P$, i.e., let $U$ be a smooth manifold where the group $G$ acts on the left, i.e., there is an action $\Sigma \colon G \times U \to U$.  Then, assuming the $P/G$ is a smooth manifold, consider the bundle over $P/G$ with standard fiber $U$ given by $\tilde{U} = (P \times U)/G = P \times_G U$.   We shall denote the pull--back $\rho^*\tilde{U}$ of $\tilde{U}$ to $P$ or to $T^*P$ with the same symbol if there is no risk of confusion.   Notice that $\rho^*\tilde{U} \cong P \times U$, however we will keep the notation $\tilde{U}$ whenever it will be convenient to indicate the we are considering a bundle over $P$.



\section{The reduction of  the Lagrange--Pontryagin principle for a class of optimal control problems with symmetry}

Assuming that $(P,C,\Gamma, L)$ is an OCP with symmetry $G$, then
the Lagrange--Pontryagin's functional
$$
\mathbb{S}(x,p,u)=\int_0 ^T [L(x,u) + \langle
p,\dot{x}-\Gamma(x,u)\rangle]dt
$$ 
defined on the space $\mathcal{C}(M)$ of smooth curves $\gamma(t)=(x(t),p(t),u(t))$ on the Pontryagin's bundle $M$, is invariant with respect to the induced action of the group $G$ on such space of curves, $g\cdot \gamma (t) = (g\cdot x(t), g \cdot p(t), g\cdot u(t))$:
$$ \mathbb{S} (g \cdot \gamma ) = \mathbb{S} (\gamma ), \qquad \forall \gamma \in \mathcal{C}(M) .$$
   If we were dealing with an specific optimal control problem with fixed endpoints $x_0$, $x_T$, we will assume that they are fixed points for the action of $G$, in such a way that the group $G$ acts on the space $\mathcal{C}_{x_0,x_T}(M)$.   Alternatively in the presence of symmetries, the optimal control problem with fixed endpoints $x_0$, $x_T$ can be replaced by a slightly generalized version where we consider endpoints belonging to the sub manifolds $G x_0$ and $G x_T$ respectively.  In both cases the group $G$ acts naturally in the corresponding space of curves.  
In any of the previous circumstances, we will have that the Lagrange--Pontryagin functional $\mathbb{S}$ will descend to the quotient space of equivalence classes of curves with respect to the action of the group $G$, $\mathcal{C}(M)/G$ and we will call it the reduced Lagrange--Pontryagin functional.     If we denote such functional by $[\mathbb{S}]_G$, then we will like to characterize its extremal curves because there is a natural one--to--one correspondence between extremal curves of $[\mathbb{S}]_G$ and equivalence classes of extremal curves of $\mathbb{S}$.   The problem we will study in the next paragraphs is the computation of the critical points of $[\mathbb{S}]_G$ by using the geometry of the quotient of the Pontryagin bundle $M$ by $G$.   The explicit expression of the Euler--Lagrange's equations $\delta [\mathbb{S}]_G = 0$  in terms of geometrical objects on $M/G$ will be called the reduced Lagrange--Pontryagin's equations and because they will correspond to the reduced expression of the first order differential conditions of the PMP, we will also call them the
reduced (first order conditions of the) Pontryagin Maximum Principle for the optimal control system $(P,C,\Gamma, L)$.
 In order to work them out we will need first the description of various quotient bundles related to a reduced optimal control problem.  This is the task of the next section.


\subsection{Description of some canonical reduced bundles}\label{reduced_bundles}

We will need first to describe in a convenient form a few bundles that will show up along the description of both, the reduction of the Lagrange--Pontryagin variational principle and the reduction of the presymplectic Dirac structure $D_\Omega$.
As we said, we will assume in what follows that the action of the Lie group $G$ on $P$ is free and proper, hence the quotient space $P/G$ is a smooth manifold and the canonical projection map $\rho_P \colon P \to P/G$ is a submersion, i.e., we are assuming that $P$ is a principal fiber bundle over $P/G$ with structure group $G$.  Orbits of $G$ on $P$ will be denoted either by $G\cdot x$ or $[x] = \{ gx \mid g\in G \}$.  Thus points $z \in P/G$ will be just $z = [x]$, with $\rho_P (x) = z$.  Local coordinates on $P/G$ will be denoted by $z^\alpha$, $\alpha = 1, \ldots, s$.

\subsubsection{The adjoint and coadjoint bundles, $\tilde{\goth g}$ and $\tilde{\goth g}^*$ of a principal bundle $P(G,P/G)$}
Given a principal fiber bundle with total space $P$ and structure group the Lie group $G$, we can construct two canonical associated bundles  the adjoint and the coadjoint bundle $\mathrm{Ad} P$ and $ \mathrm{Ad}^* P$. They are the associated bundles corresponding to the adjoint action of $G$ on the Lie algebra $\goth g$ of $G$ and the coadjoint action of $G$ on the dual of the Lie algebra $\goth g^*$.  The adjoint bundle is the vector bundle over $P/G$ with fiber the linear space $\goth g$ and structure group $G$, this is $\mathrm{Ad}P = P \times_G \goth g$, where $P \times_G \goth g = (P \times \goth g)/G$ and the group $G$ acts on the left on $P \times \goth g$ as $g\cdot (x, \xi ) = (xg^{-1}, Ad_g \xi )$.  The adjoint bundle will be also denoted in what follows by $\tilde{\goth g}$ and its points will be denoted by $(z,\zeta )$.
In a similar way we will define the coadjoint bundle as the associated bundle to the principal bundle $P(G,P/G)$ with respect to the coadjoint action of $G$ on the dual of the Lie algebra $\goth g^*$, thus $\mathrm{Ad}^*P = P \times_G \goth g^*$.  Notice that because the coadjoint action is the dual action of $G$ on $\goth g$, then the bundles $\mathrm{Ad}^*P$ is the dual bundle to $\mathrm{Ad}P$, thus it makes sense to denote it by $\tilde{\goth g}^*$.

\subsubsection{The principal bundles $TP$ and $T^*P$}
We will stablish now some notations and identities for $TP$ and $T ^*P$. 
Because $G$ acts freely and properly on $P$, it also acts freely and properly over $TP$.
We will denote this action by $\Phi_{TP} \colon G \times TP \to TP$, $\Phi_{TP}(g,x,v_x) = (gx, T\Phi_g(x)v_x)$, $g\in G$, $x\in P$, $v_x\in T_xP$,  or simply by $g(x,v_x) = (gx, gv_x)$ if there is no risk of confusion.
Then $\rho_{TP} \colon TP \to TP/G$ is a principal bundle with fiber $G$.  It is worth mention that $TP$ is also a principal fiber bundle over $T(P/G)$ with fiber $TG$.
We can see it in two steps. First, we consider the tangent map to the projection $\rho_P$, $T\rho_P \colon TP \rightarrow T(P/G)$. It maps $(x,v_x)\in T_x P$ into $(z,v_z)$ where $v_z \in T_z(P/G)$ is given by $v_z = T\rho_P(x)v_x$.
$$\begin{array}{cc} {\bfig\xymatrix{TP\ar[r]^{T\rho}\ar[d]_{\tau_P}
& T(P/G)\ar[d]^{\tau_{P/G}} \\ P\ar[r]_{\rho}& P/G&&}\efig}
\end{array}$$
Second, consider the Lie group $TG$. We can use left translations to identify it with $G\times \goth g$ by means of the map $\Lambda \colon TG \to G \times \goth g$, $\Lambda (g, v_g) = (g, \xi = TL_g^{-1}v_g)$.  
The natural composition law of $TG$ is given by $(g,v_g)\cdot (h,v_h) = (gh, v_g \cdot v_h)$ where $v_g  \cdot v_h = 
d(g(t)h(t))/dt |_{t=0}$ and $v_g  = dg(t)/
dt |_{t=0}$, $v_h = d h(t)/dt |_{t=0}$ respectively and the composition law induced on $G \times \goth g$ becomes $(g,\xi).(h,\zeta)=(g.h,Ad_h \xi +\zeta)$, for all $(g,\xi), (h,\zeta)\in G\times \goth g$.
Moreover there is a natural action of $TG$ on $TP$ given as: $$(g,v_g).(x,v_x)=(g.x,\dfrac{d}{dt}
g(t)x(t)|_{t=0})\in T_{gx} P.$$ Notice that $\dfrac{d}{dt} g(t)
x(t)|_{t=0}=\dfrac{d}{dt} g(t) g^{-1} g x(t)|_{t=0}=(Ad_g \xi)_P
(gx)+g_{*}v_x$ or in other words:
$$
(g,\xi).(x,v_x)=(g.x,g_{*}v_x+(Ad_g \xi)_P (g.x))
$$
In a similar way the right action of $G$ on $P$ can be lifted to a right action of $G$ on $T^*P$ as $(x,\alpha_x)\cdot g = (xg, T\Phi_g^*(x)\alpha_x )$, $g \in G$, $x\in P$, $\alpha_x \in T_x^*P$.    The action of $G$ on $T^*P$ is free and proper hence $T^*P$ is a principal bundle over $T^*P/G$.  Moreover the cotangent bundle $T^*G$ of $G$ carries a natural Lie group structure.  This structure can be easily described by using the last characterization of the Lie group structure of $TG$ as a semi-direct product.  In fact, we can identify $T^*G$ with the cartesian product $G\times \goth g^*$ by using right translations, i.e. $\Lambda^* \colon T^*G \to G \times \goth g^*$, $\Lambda (g, \alpha_g) = (g, R^*\alpha_g)$.    Hence because $G$ acts naturally on $\goth g^*$ by the coadjoint representation, we can define the semidirect product with composition law $(g,\mu)\cdot (h, \nu ) = (gh, Ad_h^*\mu + \nu)$.   With this group structure $T^*G$ acts on $T^*P$ as $(x,\alpha_x)\cdot (g, \mu ) = (xg, )$.

\subsubsection{The associated bundle $\tilde{N}$ determined by a subgroup $K$ of the structure group $H$ of a principal bundle $X$}
Some features of the description and properties of the quotient bundle $TP/G$ are going to appear repeatedly on what follows,
thus it is worth to spend some lines commenting the general framework to such problem.  Consider $X$ to be a right principal bundle with structure group $H$, base space $X/H$ and projection map $\rho_{X,H} \colon X \to X/H$.  Let $K$ be a subgroup of $H$ and $N = H/K$ the corresponding homogeneous $H$--space.  Notice that $H$ acts naturally on the left on $N$ by $h\cdot (h'K)  = (hh')K$, for all $h'K$ a coset in $N$.  Then $X$ is also a principal bundle with structure group $K$, base space $X/K$ and projection map $\rho_{X,K}\colon X \to X/K$.  There is also a natural projection map $\eta \colon X/K \to X/H$ given by $xK \mapsto \eta(xK) = xH$.   We observe that the previous map $\eta$ defines a bundle over $X/H$ with standard fiber $N$.  Moreover this projection map is nothing else but the associated bundle to the $H$ principal bundle $X$ by the action of $H$ on $N$ above, this is $X \times_H N \cong X/K$, where $X \times_H N$ denotes the quotient space of $X \times N$ with respect to the left action of $H$ on it defined by $h\cdot (x,h'K) = (xh^{-1}, (hh')K)$.   We can denote by $\Pi \colon X \times N \to X \times_H N$ the map sending each point $(x, h'K)$ into the corresponding equivalence class $[x,h'K]$ (the diagrams below illustrate some of these maps).  

$$\begin{array}{cc} {\bfig\xymatrix{X\times N\ar[r]^{p_X}\ar[d]_{\Pi}
& X \ar[d]^{\rho_{P,H}} \\ X\times_H N\ar[r]_{\eta}& X/H&&}\efig}
\end{array}  \quad \begin{array}{cc} {\bfig\xymatrix{X\ar[r]^{\rho_{X,H}}\ar[d]_{\rho_{X,K}}
& X/H \ar[d]^{\cong} \\ X/K\ar[r]_{\eta}& X/H&&}\efig}
\end{array}$$
In fact, it is clear that the canonical map $X \times N \to X/K$ defined by $(x,h'K) \mapsto (xh')K$ descends to a map $X \times_H N \to X/K$, that provides the identification between both spaces.   In the particular instance that $H$ is a semidirect product of two groups $K$ and $N$ 
where $N$ is a normal subgroup of $H$, then we conclude that the quotient space $X/K$ is a bundle with standard fiber the group $N$ and structure group $K$ with respect to the action of $K$ on $N$ provided by the semidirect structure, i.e., $K$ acts on $N$ by the restriction to $N$ of the inner automorphism $h \mapsto h k h^{-1}$.   Finally if, $N$ is a vector space with the natural abelian group structure, then the quotient space $X/K$ is the associated vector bundle to $X \to X/H$ with structure group $K$ and fiber $N$.  

In the situation we were considering above the principal bundle $X$ was the space $TP$ and the structural group $H$ of the principal bundle was $TG$.  The structure group $TG$ is a semidirect product of the subgroups $G$ and the abelian group $\goth g$ where the action of $G$ on $\goth g$ is given by the adjoint action.   Then, using the remarks above we conclude that $TP/G$ is the associated bundle over $TP/TG = T(P/G)$ with fiber $\goth g$ with respect to the adjoint action of $G$ on $\goth g$, $TP  \times_{TG} \goth g$.  To complete the description of this bundle we shall only notice that the associated bundle $TP \times_{TG} \goth g$ is just the pullback of the adjoint bundle $\tau_{P/G}(\tilde{\goth g})$.

 \subsubsection{The reduced tangent bundle $TP/G$ or Atiyah bundle}

The bundle $TP/G$ can be identified with the pull-back to $T(P/G)$ of the adjoint bundle $\tilde{\goth g}\rightarrow P/G$, and with some slight abuse of notation will be denoted as $T(P/G)\oplus \tilde{\goth g}$.
This elements will be of the form $[x,v_x]=(z,v_z,\zeta)$ with
$(z,v_z)\in T(P/G)$ and $ \zeta\in \goth g$.

If we choose a principal connection $A$ on $P\rightarrow P/G$,
then we have an explicit identification of $T(P/G) \oplus
\tilde{\goth g}$ and $TP/G$. The map given by:
$$ \alpha : TP/G \rightarrow T(P/G)\oplus \tilde{\goth g}, ~\alpha([x,v_x])=(T\rho_x (v_x),A_x(v_x))$$
where $T\rho_x:T_x P\rightarrow T_{G\cdot x}(P/G)$ is the tangent map to the projection map $\rho: P\rightarrow
P/G$. The dual of this map allows also for an explicit identification of $T^*P/G \cong T^*(P/G)\oplus \tilde{\goth g}^* $.
We have shown that $TP/G \cong T(P/G) \oplus \goth g$ is a bundle over $P/G$.

 \subsubsection{The reduced cotangent bundle $T^*P/G$ or Weinstein space} \cite{We78}

 The space $T^*P/G$ has a canonical bundle structure over
$T^*(P/G)$. The projection map
 $\tilde{\pi}_P : T^*P/G \rightarrow T^*(P/G)$ is defined as:
 $\tilde{\pi}_P ([x,p_x])=[x]=G\cdot x$.
The bundle $T^*P/G \rightarrow T^*(P/G)$ is the pull-back of the coadjoint bundle $\tilde{\goth g}$ to $T^*(P/G)$ along the projection map $\pi_{P/G} \colon T^*(P/G)\to P/G$.   Again with some slight abuse of notation we will denote it for $ T^*(P/G)\oplus \tilde{\goth
g}^*$.
Elements on  $T^*P/G$ will be denoted as $[x,p_x]=(z,p_z,\mu)$
where $(z, p_z)\in T^*(P/G)$, $z^a ,\: a=1,\ldots ,m$ are local coordinates on $P/G$, and $\mu\in \goth g ^*$.

As we have seen in the previous section both in the geometrical and in the variational picture appears the Pontryagin's bundle $M$. Thus the quotient space $M/G$ plays a fundamental role for the reduction procedure in either treatment of optimal control problems with symmetries.

\subsubsection{The reduced control bundle $C/G$}
If $G$ is a symmetry group of the optimal control system $(P,C,\Gamma, L)$, then the quotient space $C/G$ fibrates over $P/G$.  In fact
the bundle $\tilde{\pi} \colon C/G \to P/G$ is nothing else but the push-forward of the bundle $C$ along the projection map
$\rho$.     Notice that the map $\tilde{\pi}:C/G \rightarrow P/G$
given by $\tilde{\pi}([x,u])=[x]$ is well-defined
and $\tilde{\pi}^{-1} [x] \cong \pi^{-1} (x )$ because of the maps $\beta_x : \pi^{-1} (x) \rightarrow \tilde{\pi}^{-1}
[x] $, $ \beta_x (u) =[x,u]$ and $\gamma_x : \tilde{\pi}^{-1} [x]
\rightarrow \pi^{-1} (x)$, $\gamma_x([x,u])=(x,u)$. Then,
$\beta_x \circ \gamma_x = id_{\pi^{-1}(x)}$. Elements on
$\tilde{C}=C/G $ will be of the form $[x,u]=(z,u)$ with $z \in P/G$.  
In what follows and in order to make more transparent some constructions relative to quotient spaces, we will assume that 
the control bundle $\pi \colon C \to P$ is the pull--back to $P$ of the associated bundle to the principal bundle $\rho \colon P \to P/G$ with fiber the space $U$ where the group $G$ acts on the left.  With the notation above such associated bundle will be denoted by $\tilde{U} \to P/G$ and its pull--back to $P$, $\rho^*(\tilde{U})$ is simply $C = P\times U$.   Hence we will consider that the control bundle is trivial, however the action of the symmetry group $G$ on both factors is such that the quotient bundle is not necessarily so and becomes $C/G = \tilde{U}$ as a bundle over $P/G$. 

\subsubsection{The reduced Pontryagin bundle $M / G = (T^*P \times_P C) / G$}\label{MG}
The Pontryagin bundle $M = T^*P \times_P C$ is a bundle over $P$.  Moreover if $G$ is a symmetry group of the corresponding control system the action of $G$ lifts to $M$ and fibrates over $P$, hence the quotient space $M/G$ is a bundle over $P/G$, and it can easily be identified with the bundle $T^*(P/G)\oplus \tilde{\goth
g }^*\oplus \tilde{U}$ over $P/G$. Then the elements of the reduced Pontryagin bundle will be of the form:
$[x,p,u]=((z,p_z),\mu,u)$ where $(z,p_z)\in T^*(P/G)$, $\mu \in
\goth g^*$ and $u\in \tilde{U}$.

$$\begin{array}{cc} {\bfig\xymatrix{TP\ar[r]^{T\rho}\ar[d]_{\tau_P}
& T(P/G)\ar[d]^{\tau_{P/G}} \\ P\ar[r]_{\rho}& P/G&&}\efig}
\end{array}$$


\subsection{The reduced Lagrange--Pontryagin principle}

Now we can write the explicit form of Lagrange-Pontryagin principle.
The lagrangian $L$ being invariant defines a function $l \colon \tilde{U} = C/G \to \mathbb{R}$,
by $l(z,u) = L(x,u)$, $(z,u) \in \tilde{U}$ and $x\in P$ such that $\rho_P(x) = z$.
The dynamical vector field $\Gamma \colon C = P \times U \to TP$, takes the form
$$\Gamma(x,u)=\tilde{\Gamma}(z,u) \oplus \tilde{\gamma}(z,u) $$
with $\tilde{\Gamma}(z,u)\in T(P/G)$ and $\tilde{\gamma}(x,u)\in \tilde{\goth g}$.
In fact $T_x P= T_z (P/G) \oplus {\goth g}$, hence $\Gamma(x,u)\in T_xP$ decomposes as an element $\tilde{\Gamma}(x,u)$ in $T_z(P/G)$ and another vector $\gamma(x,u)$ in $\goth g$.   Notice that if $\gamma(x,u)=A_x(\Gamma(x,u))$ where $A$ is a principal connection on $P$, then $\gamma(gx,gu)=A_{gx} (\Gamma(gx,gu))=A_{gx}(g_{*} \Gamma(x,u))=(L^*_g A)_x(\Gamma(x,u)=Ad_g(A_x (\Gamma(x,u))=Ad_g \gamma(x,u)$, then $\tilde{\gamma}(z,u)=[\gamma(x,u)]$, and $\gamma$ defines an element on the bundle $\tilde{\goth g}$ over $P/G$.
Moreover, if we compute $\tilde{\Gamma}(gx,gu) = T\rho_P (\Gamma (gx,gu)) = T\rho_P (g_*\Gamma(x,u)) = \tilde{\Gamma}(x,u)$, thus $\tilde{\Gamma}$ induces a map (denoted with the same symbol) $\tilde{\Gamma} \colon \tilde{U} \to T(P/G)$, $\tilde{\Gamma}(z,u) \in T_z(P/G)$.

Using the previous observations we can write then the density $\mathbb{L}(x,u,p)=L(x,u)+\langle p, \dot{x} - \Gamma(x,u) \rangle$ of the Lagrange-Pontryagin variational principle as:
 \begin{equation}\label{red_lag}
 \mathbb{L}(x,u,p)=l(z,u)+\langle p_z, \dot{z}-\tilde{\Gamma}(z,u) \rangle + \langle \mu , \xi -\tilde{\gamma}(z,u) \rangle
 \end{equation}
 where $\dot{x} \in T_xP$ is decomposed according to the identification $T_xP = T_z(P/G)\oplus {\goth g}$, and $\dot{z} = T\rho_P(\dot{x})$ and $\xi = A_x(\dot{x})$ with $A$ a principal connection on $P$.  We have also used the identification above $T_x^*P \cong T_z^*(P/G) \oplus {\goth g}^*$ and $p_x \mapsto (p_z, \mu )$.  Notice that $J_P (p_x) = \mu$ and $p_x$ is the horizontal lifting of $p_z$ with respect to the connection $A$ again.

The reduced Lagrange-Pontryagin variational principle is defined on a space of curves on $M/G$, however those curves are directly related to curves on $M$, thus the variations of such curves are not completely independent.  Because of the previous discussions, we have that $M/G=T^*P\times_G C = T^*(P/G) \oplus \tilde{\goth g}^*\oplus \tilde{U}$ with elements of the form $(z,p_z,\mu,u)$ and variations $\delta z, \delta p_z,\delta \mu,\delta u$.  Notice that $\delta z(0)=0=\delta z(T)$, because curves $x(t)$ projects down to curves $z(t)\in P/G$, $x(0)=x_0 ,\: x(T)=x_T$, and $z(0)=\rho(x_0)$, $z(T)=\rho(x_T)$.    

Moreover, given a principal connection $A$ on $P$, we can always write $x(t)=g(t)x^h (t)$ where $x^h (t)$ is the unique horizontal curve on $P$ starting at $x_0$ and projecting onto $z(t)=\rho(x(t))$.  Thus $x_0=g(0)x_h (0)= g(0)x_0$, then $g(0)=e$ and similarly $g(T)=e$.
If we write $\dot{x}(t)=\dot{g}(t)x^h (t) + g(t) \dot{x}^h (t)= \dot {g}(t) g^{-1} (t) g(t) x^h (t) + g(t) \dot{x}^h (t)=\zeta(t)_P (x(t))+g(t)\dot{x}^h (t)$, using the identification $T_xP \cong T_z (P/G) \oplus \tilde{\goth g}$, provided by $\dot{x}=(\dot{z},\zeta)$, where $\dot{x}=\dot{x}^h + \dot{x}^v$, $\rho_* \dot{x}^h = \dot{z}$, $\dot{x}^v = \zeta _P (x)$ $\zeta \in \mathcal{G}$ we get that the curve $\zeta(t) \in \goth g$ corresponds to the vertical component of the curve $\dot{x}$.
Moreover $\zeta(t)=\dot{g}(t)g^{-1}(t)$,
then: $$\delta \zeta = \delta \dot{g} g^{-1} + \dot{g} \delta
g^{-1}= \delta \dot{g} g^{-1}- \dot{g}g^{-1}\delta g
g^{-1}=
\dot{\widehat{\delta g g^{-1}}}  + \delta g g^{-1} \dot{g} g^{-1}
-\zeta \delta g g^{-1}= \dot{\eta} + [\eta,\zeta]$$ and the variations of the curve $\xi(t)$ will have the form:
$$
\delta \zeta = \dot{\eta} + [\eta,\zeta] \in \goth g
$$
for an arbitrary $\eta(t) \in \goth g$.

 \subsection{The Lagrange--Pontryagin's reduced equations}\label{pmpreduc}

The Euler-Lagrange equations defined by the reduced Lagrange-Pontryagin's principle will provide us with the differential conditions of PMP for the reduced system. We can compute them from the reduced variational principle $S_G$ on $M/G$ defined by the density eq. (\ref{red_lag}):
 $$
 S_G=\int_0 ^T [l(z,u)+\langle p_z,\dot{z}-\tilde{\Gamma}(z,u)\rangle + \langle \mu, \xi- \tilde{\gamma}(z,u)\rangle ]dt
 $$
or from the hamiltonian version given by:
\begin{equation}\label{red_funct}
S_G = \int_0 ^T [\langle p_z,\dot{z}\rangle + \langle \mu, \xi\rangle - h(z,p_z,u) ]dt
\end{equation}
where $h\colon M/G \to \mathbb{R}$ is the reduced Pontryagin Hamiltonian defined by:
$$ h (z,p_z,u) = \langle p_z, \tilde{\Gamma}(z,u) \rangle + \langle \mu, \tilde{\gamma}(z,u) \rangle - l(z,u) .$$

A standard analysis parallel to that in \cite{Yo09} will provide us the Euler--Lagrange equations of the functional
eq. (\ref{red_funct}), $\delta S_G = 0$:
\begin{equation}\label{RHP} \left\{ \begin{array}{l} \frac{dz}{dt} = \dfrac{\partial h}{\partial p_z}, \\\\
\frac{Dp_z}{dt} = -\dfrac{\partial h}{\partial z} - F_A(\dot{z},\cdot ), \\\\
\frac{D\mu}{dt} = - ad_{\xi}^*\mu,    \\\\
\xi = \dfrac{\partial h}{\partial \mu},    \\\\
\dfrac{\partial h}{\partial u} =  0. \end{array} \right.
\end{equation}
that will be called the reduced PMP.


\section{Reduction of the Dirac structure of optimal control problem with symmetry}

The Dirac structure $D_\Omega$ defined on the Pontryagin bundle $M$ of an optimal control problem is a sub bundle of the bundle $TM \oplus T^*M$ over $M$.  It the optimal control problem has a symmetry group $G$, then we can therefore reduce $D_\Omega$ by taking the quotient space $D_\Omega/G$ and  consider it again as a sub bundle of $(TM \oplus T^*M)/G$.   That this idea is consistent is the content of the following simple observation:

\begin{proposition}  Under the same hypotheses as in section \ref{symmetric_ocp},
the Dirac structure associated to the presymplectic form $\Omega$, $D_\Omega \subset TM \oplus T^*M$ is $G$-invariant.
\end{proposition}

\begin{proof}
We notice that the action of  $G$ on $M$ lifts canonically both to $TM$ and $T^* M$.
 Then $(X, \alpha)\in D_\Omega$, implies that $(g_{*}X,g^{*-1}\alpha)\in D_\Omega$, where $gx$ denotes the action of a group element $g\in G$ on $x\in M$.
In fact if  $(X, \alpha)\in D_\Omega$,
then $\Omega(X,Y)=\langle \alpha, Y \rangle$, $\forall Y \in
TM$ and therefore $(g^* \Omega)(X,Y)=g^* \langle \alpha, Y
\rangle = \langle g^{*-1}\alpha, g Y \rangle$.

On the other hand,
$(g^\ast \Omega)(X,Y)=\Omega(g_* X,g_* Y)$ and therefore
$(g_*X,g^{*-1} \alpha)\in D_\Omega$.
\end{proof}

\begin{definition}
The sub bundle $[D_\Omega]_G \subset TM/G \oplus T^*M/G$ over $M/G$  is  called the \textsf{Dirac reduction}  of $D_\Omega$.
\end{definition}

A simple calculation using the identification $M/G \cong T^*(P/G) \oplus \tilde{\goth g}^* \times \tilde{U}$ (as bundles over $P/G$) discussed in \ref{MG}, shows that 
\begin{equation}\label{TMG}
TM/G \cong T(T^\ast P \times_P \tilde{U})/G \cong TT^\ast (P/G) \oplus (\tilde{\goth
g}^\ast \oplus \tilde{U}) \times (\tilde{U} \oplus \tilde{V}) 
\end{equation} 
and
\begin{equation}\label{TstarMG}
T^\ast M/G \cong T^\ast(T^\ast P \times \tilde{U})/G \cong T^\ast
T^\ast (P/G)\oplus (\tilde{\goth g}^\ast \oplus \tilde{U}) \times
(\tilde{U}^\ast \oplus \tilde{V}^\ast) , 
\end{equation}
where we assume as in section \ref{reduced_bundles} that $P \to P/G$ is a principal bundle with fiber $G$ and that the bundle of the controls $\pi \colon C \to P$ is the pullback to $P$ of a bundle associated to  $\rho_P\colon P \to
P/G$ with standard fiber $U$, i.e., we have an action of  $G$ on the space $U$ and we consider the bundle on
$P/G$ defined by $(P \times U)/G= P\times_G U$ and we denote this bundle by $\tilde{U}$. 
We see that $C=\rho_P^\ast \tilde{U}$ is a bundle over $P$ with fiber $U$ and group structure $G$.
With a little abuse of notation we shall denote even with $\tilde{U}$ the pull-back of this bundle to $T^*(P/G)$, $T(P/G)$, etc.
In both equations (\ref{TMG}) and (\ref{TstarMG}), we denote, as usual, by $\tilde{\goth g}$, $\tilde{\goth g}^*$ the associated bundles to $\goth g$,  $\goth g^*$ by the adjoint and coadjoint actions respectively, and by $\tilde{V}$ and  $\tilde{V}^*$ the associated bundles to $V=\goth g \oplus \goth g^*$ and  $V^*= \goth g^* \oplus \goth g$ by the corresponding induced actions respectively (notice that we have identified $(\tilde{\goth g}^*)^*=\tilde{\goth g}$ and $(\tilde{V}^*)^*=\tilde{V}$).

Again another calculation shows that:
\begin{equation}\label{DOmegaG}
[D_\Omega]_G \subset  TT^\ast (P/G) \oplus T^\ast T^\ast (P/G)
\oplus (\tilde{\goth g}^* \oplus \tilde{U}) \times (\tilde{U}
\oplus \tilde{U}^\ast \oplus \tilde{V} \oplus \tilde{V}^\ast).
\end{equation}

Notice that  $[D_\Omega]_G=D_{\Omega}/G$ does not define a Dirac structure on the reduced space $M/G$ because is not a sub bundle of $T (M/G) \oplus T^*(M/G)$.   In fact $TM/G \cong T(M/G)
\oplus \tilde{\goth g} \cong T(T^*P \times \tilde{U}/G) \oplus \tilde{\goth g} \cong $ and $T^\ast M/G \cong T^\ast(M/G) \oplus
\tilde{\goth g}^*$, then $[D_\Omega]_G ~\subset  T(M/G)
\oplus T^\ast(M/G) \oplus \tilde{\goth g} \oplus  \tilde{\goth g}^*$. We stress that  $[D_\Omega]_G \rightarrow M/G$ is a bundle. We call this bundle a fiber Dirac structure.

As it follows from the expressions above, dealing directly with the reduced Dirac structure $[D_\Omega]_G$ is quite involved.  
However, the presymplectic form $\Omega$ is the pull--back to $M$ with respect to the projection map $\rp_P \colon M\to T^*P$ of the natural  symplectic structure $\omega_P$ on the cotangent bundle $T^*P$, $\Omega =\rp_P^* \omega_P$.   Then the Dirac structure $D_{\Omega}$ will be related to the canonical Dirac structure on $T^*P$ associated to the symplectic form $\omega_P$ in a similar way.  This simple observation that will help enormously on this task.  The precise definition of such relation is given by the notion of pull--back and push--forward of Dirac structures \cite{Yo09}. 

\begin{definition}[Push--forward and pull--back of Dirac structures]
Let  $\psi \colon N \to N' $ a differentiable map, and let $\mathfrak{Dir}(N)$ and $\mathfrak{Dir}(N')$ be the sets of Dirac structures on $N$ and $N'$ respectively.   The forward map $\mathcal{F}\psi \colon \mathfrak{Dir}(N) \to \mathfrak{Dir}(N')$
can be defined for $D \in \mathfrak{Dir}(N)$ by:
$$
\mathcal{F}\psi (D) = \{ (\psi_*(u),\alpha)| u \in TN, \alpha\in TN'^*; (u,\psi^*(\alpha))\in D \} ,
$$
and the backward map $\mathcal{B}\psi \colon \mathfrak{Dir}(N') \to \mathfrak{Dir}(N)$ can be defined for $D' \in \mathfrak{Dir}(N')$ by:
$$
\mathcal{B} \psi(D') = \{ (v,\psi^{*} (\beta)) | v \in TN, \beta \in TN'^*; \, (\psi_*(v),\beta) \in D' \}.
$$
Using these mappings we can define a pullback and a push forward of Dirac structures along a map $\psi$ as: 
$$ 
\psi_* D = \mathcal{F}\psi (D)
$$
$$
\psi^* D' = \mathcal{B}\psi(D')
$$ 
\end{definition}

Now we see immediately that the pullback $\rp_P^* D_P$ of the canonical Dirac structure on $T^*P$, i.e. $D_P = D_{\omega_P} \subset
T(T^\ast P)\oplus T^\ast (T^\ast P)$ defined by the canonical symplectic structure $\omega_P$ in $T^\ast P$, with respect to the projection map $\rp_P \colon M\to T^*P$ is $\mathcal{B} \rp_{P}(D_{P}) = \{ (X,  \rp^*_{P} \alpha) \in TM \oplus
T^* M  \mid  ({\rp_P}_* X, \alpha) \in D_P \}$ and $\mathcal{B}\mathsf{\rp}_{P}(D_{P}) = D_\Omega$ and similarly, the push--forward of the Dirac structure $D_\Omega$ with respect to the projection map $\rp_P$ is $D_P$, i.e,. $\mathcal{B}\rp_P (D_\Omega ) = D_P$.  

The main idea is brought from the description of Marsden and Yoshimura \cite {Yo09} of implicit lagrangian systems. Therefore we will name this description as the implicit description of regular optimal control problems in terms of Dirac's structures. 

Moreover we can consider directly the Dirac structure on the quotient space  $M/K=\tilde{M}$, associated to the symplectic structure $\tilde{\Omega}$, i.e.,
$D_{\tilde{\Omega}} \subset ~T \tilde{M} \oplus ~T^* \tilde{M}$.

The application $ p: M \rightarrow \tilde{M}$ induces an application $Tp: TM
\rightarrow  T \tilde{M}$ and $D_\Omega$ is the backward with respect to $Tp$ of $D_{\tilde{\Omega}}$.

In fact, following the notation of \cite{Yo09} we can introduce the backward map, given by
$BTp(D_{\tilde{\Omega}})=\{(X, p^* \tilde{\alpha})\in TM \oplus
T^* M |~(p_\ast X, \tilde{\alpha}) \in D_{\tilde{\Omega}}\}$ and we can see that $BTp(D_{\tilde{\Omega}})=D_\Omega$.

We shall denote the relation too by  $D_\Omega/K:=
D_{\tilde{\Omega}}$ where $K={\textrm Ker} Tp$ and we shall say that the presymplectic Dirac structure $D_\Omega$ projects on the symplectic Dirac structure $D_{\tilde{\Omega}}$.

Finally we observe that the Lie -Dirac reduction of the Dirac structure $D_{\tilde{\Omega}}$ give us another fiber Dirac structure $[D_{\tilde{\Omega}}]_G$ on $T(\tilde{M}/G)$,
i.e.:
$$[D_{\tilde{\Omega}}]_G:= D_{\tilde{\Omega}}/G  \subset  T(\tilde{M}/G)
\oplus T^\ast(\tilde{M}/G) \oplus \tilde{\goth g} \oplus
\tilde{\goth g}^*.$$

It is easy to prove that also the forward map of $[D_\Omega]_G$ with respect to the projection  $p_P: M\rightarrow T^*P$ is the reduced Dirac structure on $T^* P$.

The application $p \colon M \to \tilde{M}$ induces an application
$p_G \colon M/G \to \tilde{M}/G$. It is easy to have that
$Bp_G[D_{\tilde{\Omega}}]_G=[D_{\Omega}]_G$, or roughly speaking, the Dirac reduction of the Dirac structure defined by the symplectic form  $\Omega$ is the pull-back of the Dirac reduction of the Dirac structure induced by the symplectic form $\tilde{\Omega}$, or, the fiber Dirac structure  $[D_{\Omega}]_G$ projects to the fiber Dirac structure $[D_{\tilde{\Omega}}]_G$.


 \subsection{Reduced Dirac equations of an optimal control problem with symmetry}
\label{REOCP}
 Once the reduced structure is defined, we would have to present the associated reduced equations:
 \begin{equation}
 ([X],[dH])\in [D_{\Omega}]_G
 \end{equation}
 where the symbol $[x]$ indicates the reduced system.
Determine these equations directly is possible but complicated.
To render this calculation more easily we make the following observation:
the presymplectic structure $\Omega$ on $M$ is given by the pull-back with respect to the projection $p_P: M\rightarrow T^*P$ of the natural symplectic structure $\omega$ on the cotangent bundle $T^*P$: $\Omega=p{_P}^* \omega$. Then $\Omega$ doesn't contains information on the control bundle $C$.
In other words if we consider the Dirac structure on $T^*P$ associated to the natural symplectic form $\omega_P$:
$$
D_P\subset T(T^*P)\oplus T^*(T^*P)
$$
it is easy to prove that the projection of $[D_\Omega]_G$ with respect to the bundle  $\tilde{U} \oplus \tilde{U}^\ast$, projects on the cotangent reduction of the canonical Dirac structure on
$T^\ast P$ because the factors  $\tilde{U}$ and $\tilde{U}^\ast$ don't appear in the definition of $\Omega$.

We shall denote this projection by  $\Pi \colon [D_\Omega]_G \to [D_{ P}]_G$.
On the other hand, we can collect here the results obtained by Marsden et al regarding the cotangent bundle reduction of the Dirac structure $D_P$. Thus we have:
$$[D_{P}]_G \subset  TT^\ast (P/G) \oplus T^\ast T^\ast (P/G)
\oplus (\tilde{\goth g}^* \times (\tilde{V} \oplus
\tilde{V}^\ast)),$$ and in local form

$$[D_{P}]_G=\{(z,p_z,\mu;v_z,v_{p_z};\pi_z,\pi_{p_z};\xi,
v_{\mu};\alpha_\xi,\alpha_{v_{\mu}})~|~$$$$| \langle \pi_z,\delta
z \rangle+\langle \delta p_z, \pi_{p_z} \rangle+\langle
\alpha_\xi, \xi \rangle+\langle \alpha_{v_\mu},v_{\mu} \rangle
=[\Omega]_{G~(z,p_z,\mu)}(v_z,v_{p_z},\xi,v_\mu)(\delta z,\delta
p_z,\zeta,\delta \mu ),$$
$$ \forall (\delta z,\delta
p_z,\zeta,\delta \mu ) \in T_{(z,p_z)}T^\ast (P/G)\times
\tilde{V}\}.$$


\subsection{The Dirac reduction and the reduced Pontryagin's Maximum Principle}
In the section \S \ref {pmpreduc} obtained the equations of the reduced PMP for Optimal Control systems with symmetry, based on calculating the differential one of the reduction of Lagrange-Pontryagin principle. On the other hand in the previous sections, we show how a optimal control problem could be described in two different ways across Dirac's structures: as a presymplectic problem or as an implicit problem with Dirac structure $D _ {T ^\ast P} =D_P$.
We saw also that both Dirac structures are related, being the second one a projection of the first one.

Finally also we saw that in presence of symmetry we can reduce both Dirac structures and the corresponding reduced Dirac structures are related again by projection.

Therefore $[D_\Omega]_G$ and $[D_P]_G$ are such that $[D_P]_G$
is the projection of $[D_\Omega]_G$ with respect to $\tilde{V}
\oplus \tilde{V}^\ast$. The reduced Dirac structure $[D_P]_G$ is the Dirac's reduction of Dirac structure cotangent in $T^\ast P$. Then if we denote for $H_G$ the Pontryagin Hamiltonian induced in $M/G=T^\ast(P/G) \oplus
\tilde{\goth g}^\ast \oplus \tilde{U}$ the quotient space of the Pontryagin bundle, we have that the sub manifold $M_1=\{
\pd{H}{u^a}=0\}$ projects to the sub manifold $[M_1]_G=\{
\pd{H_G}{u^a}=0\}$ y $M_1/G=[M_1]_G$. Therefore we will define Pontryagin's implicit equation as the dynamics defined by the equation:

 \Eq{\label{ecimplicpontry}\left([X]_G(z,p_z,\mu,u),
[dH_G]\Big |_{T[M_1]_G}(z,p_z,\mu,u)\right)\in
[D_P]_G(z,p_z;\mu),}$$ (z,p_z)\in T^\ast(P/G),~\mu \in \goth
g^\ast.$$
These equations coincide with the reduced equations of the PMP obtained previously.

\begin{theorem}\label{red_Dirac_ham}
 Given a regular optimal control system $(L,\Gamma,C \rightarrow
P)$ with symmetry group $G$, the equations of the Reduced Pontryagin Maximum Principle obtained as reduction of Lagrange-Pontryagin's variational principle, coincide with the equations of Dirac's reduction of the equations of the description of implicit Dirac of the problem of ideal control, equations \ref{ecimplicpontry}.
\end{theorem}

\dem The proof of the theorem relies on the description of Dirac's reduction of Dirac's structure cotangent $D_P$ obtained in \cite {Yo09}. We will use the above mentioned result without proving it again. This way we obtain that the structure $[D_P]_G$ for every $(z,p_z;\mu) \in T^\ast(P/G)\oplus \goth g^\ast$ has the form
(Prop. 4 \cite{Yo09}):

$$[D_{P}]_G(z,p_z;\mu)=\{(v_z,v_{p_z};\xi,
v_{\mu}),(\pi_z,\pi_{p_z};\alpha_\xi,\alpha_{v_{\mu}}) \in
T_{(z,p_z)}T^\ast (P/G)\times \tilde{V} \oplus$$ $$\oplus
T^\ast_{(z,p_z)}T^\ast (P/G)\times \tilde{V}^\ast
~|~v_z=\pi_{p_z},~~v_{p_z}+\pi_z=-F_A(v_z,.),~
\xi=\alpha_{v_{\mu}},~v_\mu+\alpha_\xi=ad^\ast_{\xi}\mu\}.$$

Then if we denote $[X]_G$ as the reduced vector field such that for every curve
 $(z(t),p_z(t);\mu(t),u(t))\in T^\ast(P/G)
\oplus \tilde{\goth g}^\ast \oplus \tilde{U}$, take the value

$$[X]_G(z(t),p_z(t);\mu(t),u(t))=(z(t),p_z(t),u(t);\dfrac{dz(t)}{dt},
\dfrac{Dp_z(t)}{Dt},\mu(t),\xi(t),\dfrac{D\mu(t)}{Dt})$$and bearing in mind that $[dH]_G=dH_G$ restrict to $T[M_1]_G=T(M_1/G)$
take the form:
$$[dH]_G \Big
|_{T[M_1]_G}=(z,p_z;\mu,u;-\dfrac{\partial h}{\partial
z},\dfrac{\partial h}{\partial p_z},
\dfrac{\partial h}{\partial \mu},0)$$
and finally:
\Eq{\label{teodemdirac}\left\{\begin{array}{ll}\dfrac{dz}{dt}=\dfrac{\partial h}{\partial p_z},\\\\
\dfrac{Dp_z}{Dt}=-\dfrac{\partial h}{\partial z}-F_A(v_z,.),\\\\
\dfrac{D\mu}{Dt}= - \mathrm{ad}^\ast_\xi\mu,\\\\
\xi=\dfrac{\partial h}{\partial \mu}.\end{array}\right.}
together with
$\dfrac{\partial h}{\partial u}=0.$
\findem

\newpage
\section{A simple example: The subriemannian geodesic problem for the Heisenberg group}
This problem is a basic example in subriemannian geometry and has been solved in many different ways. We will use it here to illustrate how the variational theory discussed above as well as the Dirac reduction procedure translate into this very fundamental example.  It is worth to point it out that in this case the hypothesis of non existence of abnormal extremals and regularity  coincides.
It is also interesting to remark that the reduction procedure provides a very simple and descriptive way to solve this problem.

The geodesic problem for the 3-dimensional real Heisenberg group $\Bbb{H}^1$ in subriemannian geometry is formulated as the optimal control problem of joining two points in $\mathbb{R}^3$ by using two non commuting vector fields $\Gamma_1=\frac{\partial}{\partial x}-\frac{y}{2} \frac{\partial}{\partial z}$ and
$\Gamma_2=\frac{\partial}{\partial x}+\frac{x}{2}\frac{\partial}{\partial z}$, which satisfy the commutation relations:
$$
[\Gamma_1,\Gamma_2]= \Gamma_3 ;~~
 [\Gamma_1,\Gamma_3]=[\Gamma_2,\Gamma_3]=0 ,
$$
and minimizing the length of the corresponding curve, i.e.: if we denote by  $q=(x,y,z)$, a point in the three dimensional space $\mathbb{R}^3$ , by $u(t) =(u_1(t),u_2(t))$ a curve in $\mathbb{R}^2$ and $\Gamma_3=\frac{\partial}{\partial z}$, and we integrate the differential equation:
\begin{equation}
\dot{q}=\Gamma(q)= u_1(t) \Gamma_{1}(q) + u_2(t) \Gamma_{2}(q) ,
\end{equation}
we want to minimize the length of the integral curves of the vector field $\Gamma$ with respect to the subriemannian metric $\eta_D$ defined on the distribution $D$ spanned by $\Gamma_1$ and $\Gamma_2$ and given by
\Eq{\label{heisenberg1}\left\{\begin{array}{ll}
\mathcal{L}=\frac{1}{2} (u_1^2+u_2^2)\\
q(0)=q_0~~q(T)=q_T .
\end{array}\right.}

The Pontryagin Hamiltonian is given by:
\begin{eqnarray}\label{heisenberg2}
H_P(q,\lambda,u)=u_1\langle \lambda, \Gamma_1\rangle +u_2\langle\lambda,\Gamma_2\rangle-\frac{1}{2}(u_1^2+u_2^2)=\\
= u_1(\lambda_1+y\lambda_3)+u_2(\lambda_2-x\lambda_3)-\frac{1}{2}(u_1^2+u_2^2)
\end{eqnarray} 
 with $\lambda\in(\lambda_1,\lambda_2,\lambda_3)\in\mathbb{R}^{3*}$.
 
 The maximization condition provides an explicit expression for the controls:
 $$
 \langle \lambda,\Gamma_a\rangle =u_a; \quad \quad a=1,2
$$
from which we obtain the hamiltonian function on $T^*\mathbb{R}^3$:
$$
\mathcal{H}_P (q,\lambda ) = \frac{1}{2}=\left( (\lambda_1+y \lambda_3)^2+(\lambda_2-x \lambda_3)^2\right) , 
$$
and the corresponding hamiltonian system is:

\begin{equation}
\left\{ \begin{array}{l}
\dot{x}=\lambda_1+y\lambda_3\\
\dot{y}=\lambda_2-x\lambda_3\\
\dot{z}=\lambda_1 y-\lambda_2 x+\lambda_3 (y^2+x^2)\\
\dot{\lambda}_1=(\lambda_2-x\lambda_3)\lambda_3\\
\dot{\lambda}_2=-(\lambda_1+y\lambda_3)\lambda_3\\
\dot{\lambda}_3=0
\end{array} \right.  .
\end{equation}
  
The previous equations are solved in various ways in literature, but here we will proceed to integrate it  by exploiting the fact that the optimal control problem above has a natural symmetry.  In order to make it explicit we will describe the optimal control problem as a problem on the 3 dimensional Heisenberg group.
For our  purposes  is better to realize the 3-dimensional real Heisenberg group $\Bbb{H}^1$ as the group of $3\times3$ upper triangular matrices of the form:
$$
\left(\begin{array}{ccc}1 & A & C \\0 & 1 & B \\0 & 0 & 1\end{array}\right)
$$ 
and the composition law of the group is matrix multiplication.  Thus, the left action $L\colon G\times G \to G$ is given by the left product of matrices. Notice that as a manifold the Heisenberg group $\Bbb{H}^1$ is diffeomorphic to $\mathbb{R}^3$.  The Heisenberg group is a nilpotent  Lie group whose Lie algebra is generated by the matrices:
\begin{equation}
\nonumber
\gamma_1=\left(\begin{array}{ccc}0 & 1 & 0 \\0 & 0 & 0 \\0 & 0 & 0\end{array}\right), \quad  \gamma_2=\left(\begin{array}{ccc}0 & 0 & 0 \\0 & 0 & 1 \\0 & 0 & 0\end{array}\right), \quad  \gamma_3=
 \left(\begin{array}{ccc}0 & 0 & 1 \\0 & 0 & 0 \\0 & 0 & 0\end{array}\right) ,
\end{equation}
with commutation relations:
\begin{equation}
\nonumber
 [\gamma_1,\gamma_2]=\gamma_3 ;~~
 [\gamma_1,\gamma_3]=[\gamma_2,\gamma_3]=0
 \end{equation}
 A simple computation shows that the exponential map is given by:
 \begin{equation}
 \nonumber
 g=\exp(a\gamma_1+b\gamma_2+c\gamma_3)= \left(\begin{array}{ccc}1 & a & c+\frac{ab}{2} \\0 & 1 & b \\0 & 0 & 1\end{array}\right) .
 \end{equation}
 Given the element $\xi = u_1 \gamma_1 + u_2 \gamma_2$ on the Lie algebra of $\Bbb{H}^1$, we can define the left invariant vector field on $\Bbb{H}^1$:
 $$
 \dot{g}=g(u_1\gamma_1+u_2\gamma_2) \quad  u_1,u_2 \in \mathbb{R} ,
 $$
that, with a suitable choice of coordinates:
 $$
 x=a; \quad y=b; \quad z=c-\frac{1}{2}xy ,
 $$
can be written as: 
\Eq{\label{heisenberg3}\left\{\begin{array}{ll}\dot{x}=u_1,\\\\
\dot{y}=u_2\\\\
\dot{z}=\frac{xu_2-yu_1}{2}\\\\
\end{array}\right.}
which is the control equation in the problem (\ref{heisenberg1}).

In this formulation we can identify the state space of the system $P=\Bbb{H}^1\cong \mathbb{R}^3$, the control bundle is the trivial bundle $C=\Bbb{H}^1\times\mathbb{R}^2$ over $\mathbb{R}^3$ and the Pontryagin bundle $M$ is given by:
$$M=T^*P\times_{P}C=T^*(\Bbb{H}^1)\times_{\Bbb{H}^1}\Bbb{H}^1\times \mathbb{R}^2)\cong \mathbb{R}^3\times\mathbb{R}^{3*}\times \mathbb{R}^2  .$$
The Pontryagin Hamiltonian $H_P \colon M\to\mathbb{R}$ is the function (\ref{heisenberg2}).
The group $\Bbb{H}^1$ acts on the left on $P$ and this action lifts trivially to $C$.  Hence the optimal control problem is invariant under $\Bbb{H}^1$.     In what follows we will indicate the symmetry Lie group of this problem either  by $\Bbb{H}^1$:
$$
H_P(g,\lambda_g,u) = h([g,\lambda_g,u]), 
$$
where $h\colon M/G \to \mathbb{R}$ is the Hamiltonian function induced by the invariant Pontryagin Hamiltonian $H_P$, $[g,\lambda,u]$ denotes the orbit of the point $(g,\lambda,u) \in M$ under the action of $G$,
and the reduced Pontryagin's bundle has the structure $M/G=T^*G/G\times U=\goth g^*\times U=\mathbb{R}^3\times \mathbb{R}^2$.
We are using left translations to identify $T^*G\cong G\times \goth g^*$, i.e.,
$(g,\lambda_g)\in T^*G\mapsto (g,\lambda)\in G\times \goth g^*$ with $\lambda=L^*_g\lambda_g\in T^*_e G=\goth g^*$,
then Pontryagin's Hamiltonian reads:
$$
H_P(g,\lambda_g,u) = \langle \lambda_g, \dot{g}(g)\rangle - \mathcal{L}(g,u) = \langle \lambda, \dot{g}(e)\rangle- l(u)
$$
where $l(u) = \mathcal{L}([g,u])$,
and the reduced hamiltonian $h$ is a function of $(\lambda,u)\in \goth g^*\times U$ alone. 
In fact we get:
$$
h(\lambda,u)=\langle \lambda, u_1 \gamma_1+ u_2 \gamma_2\rangle - \frac{u_1^2+u_2^2}{2}= \langle \theta_1,\gamma_1\rangle \lambda_1u_1+\langle\theta_2,\gamma_2\rangle \lambda_2 u_2 - \frac{u_1^2+u_2^2}{2} ,
$$
where $\lambda=\lambda_1\theta_1+\lambda_2\theta_2+\lambda_3\theta_3$ and $\{ \theta_i \}$ is the dual base of $\{ \gamma_i \}$. 
 
At this point we can use the consistency conditions (\ref{constr}):
$$\frac{\partial H_P}{\partial u_a}=\frac{\partial h}{\partial u_a}=0, \quad a=1,2$$
which implies
\begin{equation}
\nonumber
\left\{\begin{array}{l}\lambda_1=u_1\\ \lambda_2=u_2 \end{array}\right.
\end{equation}
and finally,
\begin{equation}
\label{hred}
h(\lambda)=\frac{\lambda_1^2+\lambda_2^2}{2}  .
\end{equation}

The reduced Dirac structure will be the reduction of the canonical Dirac structure on $G$, associated to the natural symplectic form on $T^*G\cong \goth g\times G$.  Here we can follow the procedure introduced in \cite{Yo07a} and described again in Thm. \ref{red_Dirac_ham}.   In the particular instance that $P$ is a Lie group $G$, we have
$$
D_G=\{(\alpha,X)\in T^*(T^*G)\oplus T(T^*G) \mid \alpha=\omega_G(.,X)\}\subset T^*(T^*G)\oplus T(T^*G) \} ,
$$
and the reduced Dirac structure is a sub bundle:
$$
\left[ D_G\right]_G\subset (T^*(T^*G)\oplus T(T^*G))/G\simeq\goth g^*\times(V\oplus V^*)
$$
where $V=\goth g\oplus\goth g^*$.
Because of the invariance of the Dirac structure, we have for each fixed $\lambda\in\goth g^*$ a Dirac structure $\left[ D_G\right]_G(\lambda)\subset V\oplus V^*$
$$
\left[D_G\right]_G(\lambda)=\{\left( (\xi,\rho),(\nu,\eta) \right)\in V\oplus V^* \mid \langle\nu,\xi\rangle+\langle\sigma,\eta\rangle=\omega_\lambda^{/G}\left((\xi,\rho), (\zeta,\sigma)\right)~ \forall(\zeta, \sigma)\in V\} ,
$$
where $\omega_{\lambda}^{/G} $is the $\lambda$-dependent  bilinear form on $V$ given by
$$
\omega_{\lambda}^{/G} \left ((\xi,\rho),(\zeta,\sigma)\right )=\langle\sigma,\xi\rangle-\langle\rho,\zeta\rangle+\langle\lambda,[\xi,\zeta]\rangle .
$$
Now we apply our main theorem, Thm. \ref{red_Dirac_ham}, that extends to optimal control problems the reduction of implicit hamiltonian systems \cite{Yo07a}.  In this case because the constraints manifold $M_1$ determined by the optimal feedback is canonically identified with $T^*G$, the reduction is that of the standard hamiltonian system $(H,T^*G,X)$ and it is  given by a triple $(h,\goth g,\chi^{/G})$ that satisfies for each fixed $\lambda\in\goth g^*$, the reduced PMP eq. (\ref{ecimplicpontry}):
\begin{equation}\label{reducheis}
\left([X]_G(\lambda),dh (\lambda)\right)\in \left[D_G\right]_G(\lambda)
\end{equation}
where $[X]_G$ is the reduced vector field:
$$
[X]_G \colon \goth g^*\to\goth g\times \left(\goth g \times \goth g^*\right)
$$
and in local coordinates:
$$
[X]_G  = \left(\lambda,\xi(\lambda),\dot{\lambda}\right) ,
$$
where $\xi(\lambda)=T_gL_{g^-1}\dot{g}$.

For the reduced hamiltonian we obtain easily that:
$$dh = \left(\lambda,0,\frac{\partial h}{\partial \mu}\right)$$
and finally, recalling the expression of (\ref{hred}), we obtain the reduced Lie-Poisson equations (\ref{reducheis}):
\begin{equation}
\nonumber
\left\{\begin{array}{l}
\xi=\frac{\partial h}{\partial \lambda}\\
\dot{\lambda}=ad^*_{\xi}\lambda
\end{array}\right.
\end{equation}
i.e.
\begin{equation}
\nonumber
\left\{\begin{array}{l}
\xi_1=\lambda_1\\
\xi_2=\lambda_2\\
\xi_3=0\\
\dot{\lambda_1}=-\lambda_3 \xi_2\\
\dot{\lambda_2}=\lambda_3 \xi1\\
\dot{\lambda_3}=0
\end{array}\right.
\end{equation}
The first three equations of this system can be interpreted as a \emph{partial Legendre transform} \cite{Yo07a}. 
Then we have to solve the reduced Poisson system on $\goth g^*$:
\begin{equation}\nonumber
\left\{\begin{array}{l}
\dot{\lambda_1}=-\lambda_3 \lambda_2\\
\dot{\lambda_2}=\lambda_3 \lambda_1\\
\dot{\lambda_3}=0
\end{array}\right.
\end{equation}
whose solution is given by:
\begin{equation} \nonumber
\left\{\begin{array}{l}
\lambda_1=\lambda_{1,0} \cos(k t)-\lambda_{2,0} \sin(k t)\\
\lambda_2=\lambda_{1'0} \sin(k t)+\lambda_{2,0} \cos (k t)\\
\lambda_3= k 
\end{array}\right.
\end{equation}
 
 If we fix a value for the reduced hamiltonian:
$$
h(\lambda)=\frac{\lambda_1^2+\lambda_2^2}{2}=\frac{1}{2}
$$
we obtain a condition for $\lambda_{1,0}$ and $\lambda_{2,0}$:
$$
\lambda_{1,0}^2+\lambda_{2,0}^2=1 .
$$
Then we can choose:
\begin{equation}
\nonumber
\left\{\begin{array}{l}
\lambda_{1,0}=\cos \theta \\
\lambda_{2,0}=\sin \theta
\end{array}\right.
\end{equation}
and the explicit expression for the solutions are:
\begin{equation} \nonumber
\left\{\begin{array}{l}
\lambda_1 (t)=\cos(\theta+k t)\\
\lambda_2 (t)=\sin(\theta+k t)\\
\lambda_3 (t)= k 
\end{array}\right.
\end{equation}
 
To solve the geodesic problem it is sufficient to reconstruct the dynamic for $g(t)$ on $G$.
Remembering that $L_{g^-1 *}\dot{g}(t)=\xi(t)$ we have:
\begin{equation}\nonumber
\left\{\begin{array}{l}
\frac{dg(t)}{dt}=T_eL_{g(t)}\xi(t)\\
g(0)=g_0\\ 
\end{array}\right.
 \end{equation}
or in components:
\begin{equation}
 \label{geo}
\left\{\begin{array}{l}
\dot{x}(t)=\lambda_1(t)\\
\dot{y}(t)=\lambda_2\\
\dot{z} (t)= x\lambda_2 -y\lambda_1
\end{array}\right.
 \end{equation}
 If we consider, for instance the geodesics starting from the origin in $\mathbb{R}^3$, we have the initial conditions:

 \begin{equation}
 \nonumber
\left\{\begin{array}{l}
x(0)=0\\
y(0)=0\\
z(0)=0
\end{array}\right.
 \end{equation}
and integrating the system (\ref{geo}) we obtain the explicit expression of the geodesics in $\mathbb{R}^3$:
\begin{equation}
\nonumber
\left\{\begin{array}{l}
x(t)=\frac{1}{k}\sin(k t+\theta)-\frac{1}{k}\sin{\theta}\\
y(0)=\frac{1}{k}\cos(k t +\theta)+\frac{1}{k} \cos{\theta}\\
z(t)=\frac{1}{k^2}\sin(kt)+\frac{t}{k}
\end{array}\right.
 \end{equation}
which describe, as it is well--known, spirals in $\mathbb{R}^3$ whose projection on the plane are circles of radius $\frac{1}{k}$ and center $(-\frac{\sin(\theta)}{k},\frac{\cos(\theta)}{k})$.




\begin{thebibliography}{xxxxx}

\bibitem[Ba07]{Ba07}  Barbero, M., Echevarr\'{\i}a Enriquez, A., Mart\'{\i}n de Diego, D. Mu\~noz-Lecanda, M., Rom\'an-Roy, N. {\it Skinner-Rusk unified formalism for optimal control systems and applications}. Journal of physics A. Mathematical and theoretical, 2007, {\bf  40}, 12071-12093.    

\bibitem[Ba08]{Ba08} Barbero, M.; Mu–oz, M. {\it Geometric approach to Pontryagin's Maximum Principle}. Acta applicandae mathematicae, 2008, 1-57.    

\bibitem[CW88]{CW88} Courant, T.; Weinstein, A.: {\it Beyond Poisson structures}, Action hamiltonienne des groupes, Troisi{\`e}me th{\'e}or{\`e}me de Lie (Lyon, 1986), vol. {\bf 27}, {\it Travaux en cours}, {\em Hermann, Paris}, 39-49, (1988).

\bibitem[Co90]{Co90} T. J. Courant: {\it Dirac manifolds}, Trans. Amer. Math. Soc., 319 ,631-661, (1990).

\bibitem[De04]{De04} M. Delgado.  Ph. D. Thesis... (2004).

\bibitem[Go79]{Go79} Gotay, M.J.; Nester, J.M.; Hinds, G.: {\em J. Math. Phys.},
{\textbf 19}, 2388, (1979).

\bibitem[Ag0*]{Agrbook} AAgrachev A.A., Sachkov Yu.L. a {\it Control Theory From The Geometric Viewpoint}, Springer Verlag, 2004

\bibitem[Ju0*]{Jurbook} V. Jurdjevic, {\it Optimal Control, Geometry and Mechanics,
Mathematical Control Theory}, J. Bailleu, J.C. Willems (ed.), 227Ð267, Springer, 1999.

\bibitem[Mr04]{Mr04} Mart\'{\i}nez. E.: {\it Reduction in optimal control theory}, Rep. Math. Phys. {\bf 53}, 79–90, (2004).

\bibitem[Me95]{Me95} Mendella, G.; Marmo, G.; Tulczyjew, W.: {\it Integrability of implicit differential equations}, {\em J. Phys. A: Math. Gen.}, {\textbf 28}, 149-163, (1995).

\bibitem[Po62]{Po62} Pontryaguine, L.S.; Boltyanskii, V.G.; GamKrelidze, R.V.; Mishchenko, E.R.: {\it The Mathematical Theory of Optimal Processes}, Interscience, N.Y., (1962).

\bibitem[Ra94]{Ra94} Rabier, P.J.; Rheinboldt, W.C.: {\it A geometric treatment of implicit differential-algebraic equations}, {\em J. Diff. Equs.}, {\textbf 109}, 110-146, (1994).

\bibitem[We78]{We78} Weinstein, A.: {\it An universal construction for particles in Yang-Mills fields}, Lett. Math. Phys., {\bf 2}, 417-420, (1978).

\bibitem{Jur} Jurdjevich {\it Hamiltonian point of view on non-Euclidean geometry and el-
liptic functions}, System Control Lett. 43 (2001), 25Ð41

\bibitem{Tor} Torres, D.F. M. {\it Conservation laws in optimal control.  Dynamics,
bifurcations, and control} (Kloster Irsee, 2001),  287--296, Lecture
Notes in Control and Inform. Sci., 273, Springer, Berlin, 2002.
(Reviewer: Corneliu Ursescu) 49K15 (93C10)

\bibitem[Yo06a]{Yo06a} Yoshimura, H. ; Marsden J.E.: {\it Dirac structures in Lagrangian mechanics. Part I}, J.Geom. and Phys.,  vol {\bf 57}, 133-156, (2006).

\bibitem[Yo06b]{Yo06b} Yoshimura, H. ; Marsden J.E.: {\it Dirac structures in Lagrangian mechanics. Part II}, J.Geom. and Phys.,  vol {\bf 57}, 209-250, (2006).

\bibitem[Yo07a]{Yo07a} Yoshimura, H. ; Marsden J.E.: {\it Reduction of Dirac structures and the Hamilton-Pontryagin principle}, Rep. on Math. Phys.,  vol {\bf 60}, 381-426, (2007).

\bibitem[Yo07b]{Yo07b} Yoshimura, H.; Marsden, J. E.: {\it Dirac structures and the Legendre transformation for implicit Lagrangian and Hamiltonian systems}, Lecture Notes in Control and Inform. Sci., {\bf 366}, 233-247 (2007).

\bibitem[Yo09]{Yo09} Yoshimura, H.; Marsden, J. E.: {\it Dirac cotangent bundle
reduction}, J. Geom. Mech., Vol. {\bf 1}, 87-158, (2009).

\end{thebibliography}
\end{document}